\documentclass[a4paper]{article}

\usepackage[margin=1in]{geometry} 

\usepackage{amsmath}
\usepackage{amsthm}
\usepackage{amssymb}
\usepackage{graphicx}
\usepackage{subcaption}
\usepackage{xcolor}
\usepackage{mathrsfs} 
\usepackage[utf8]{inputenc}
\usepackage{adjustbox}
\usepackage{enumerate}

\usepackage{tikz}
\usetikzlibrary{shapes,calc,math,backgrounds,matrix}
\usepackage{tkz-graph} 

\usepackage{hyperref}
\hypersetup{
	hidelinks,
	unicode,
	bookmarks,
	bookmarksopen,
	bookmarksdepth=2,
	pdftitle={A Note on Reconfiguration Graphs of Cliques},
	pdfauthor={Quan N. Lam, Huu-An Phan, Duc A. Hoang},
	pdfkeywords={combinatorial reconfiguration, reconfiguration graph, clique, structural properties}
}

\usepackage[capitalise, noabbrev]{cleveref}

\usepackage{tcolorbox}

\usepackage[%
backend=biber, 
bibstyle=numeric, 
citestyle=numeric-comp,
maxcitenames=3,
maxbibnames=20,
sorting=ydnt
]{biblatex}
\theoremstyle{plain}
\newtheorem{theorem}{Theorem}
\newtheorem{corollary}[theorem]{Corollary}
\newtheorem{lemma}[theorem]{Lemma}

\newtheorem{proposition}[theorem]{Proposition}

\theoremstyle{definition}

\usepackage[linesnumbered,ruled,vlined,commentsnumbered]{algorithm2e} 

\usepackage{lineno}

\newcommand{\email}[1]{\href{mailto:#1}{\texttt{#1}}} 
\newcommand{\sfR}{\mathsf{R}}
\newcommand{\sfTS}{\mathsf{TS}}
\newcommand{\sfTJ}{\mathsf{TJ}}
\newcommand{\sfTAR}{\mathsf{TAR}}
\newcommand{\TS}[2]{\sfTS_{#1}(#2)} 
\newcommand{\TJ}[2]{\sfTJ_{#1}(#2)} 
\newcommand{\TAR}[2]{\sfTAR_{#1}(#2)} 
\newcommand{\TARR}[2]{\sfTAR^{#1}(#2)} 
\newcommand{\R}[2]{\sfR_{#1}(#2)} 
\newcommand{\Uni}{\mathsf{Uni}}
\newcommand{\Int}{\mathsf{Int}}

\title{A Note on Reconfiguration Graphs of Cliques}

\author{Quan~N.~Lam$^1$\thanks{Work completed while at University of Science, Ho Chi Minh City, Vietnam.} \and Huu-An~Phan$^2$\thanks{Corresponding author} \and Duc~A.~Hoang$^3$}
\date{
	$^1$ Université Gustave Eiffel, Champs-sur-Marne, France\\
	\email{nhat-quan.lam@edu.univ-eiffel.fr}\\
	$^2$ Nanyang Technological University, Singapore\\
	\email{phanan23467@gmail.com}\\
	$^3$ VNU University of Science, Vietnam National University, Hanoi, Vietnam\\
	\email{hoanganhduc@hus.edu.vn}\\[2ex]
	\today
}

\begin{document}
\maketitle

\begin{abstract}

In a reconfiguration setting, each clique of a graph $G$ is viewed as a set of tokens placed on vertices of $G$ such that no vertex has more than one token and any two tokens are adjacent.
Three well-known reconfiguration rules have been studied in the literature: Token Jumping ($\mathsf{TJ}$), Token Sliding ($\mathsf{TS}$), and Token Addition/Removal ($\mathsf{TAR}$).
Given a graph $G$ and a reconfiguration rule $\mathsf{R} \in \{\mathsf{TS}, \mathsf{TJ}, \mathsf{TAR}\}$, a reconfiguration graph of $k$-cliques of $G$, denoted by $\mathsf{R}_k(G)$, is the graph whose vertices are cliques of $G$ of size $k$ and two vertices are adjacent if one can be obtained from the other by applying $\mathsf{R}$ exactly once.
In this paper, we initiate the study of structural properties of reconfiguration graphs of cliques, proving several interesting results primarily under $\mathsf{TS}$ and $\mathsf{TJ}$ rules.
In particular, we establish a formula relating the clique number of $G$ and that of $\mathsf{TS}_k(G)$, and bound the chromatic number of $\mathsf{TS}_k(G)$ via that of an appropriate Johnson graph.
Additionally, we present an algorithm to construct $\mathsf{TS}_{\omega(G)-1}(G)$ from $\mathsf{TJ}_{\omega(G)}(G)$ and derive structural properties of $\mathsf{TJ}_{\omega(G)}(G)$ graphs, where $\omega(G)$ denotes the clique number of $G$.
Finally, we show that $\mathsf{TS}_k(G)$ is planar whenever $G$ is planar and establish bounds on the number of $3$- and $4$-cliques based on results concerning $\mathsf{TS}_k(G)$ graphs. In particular, we prove that any planar graph $G$ with $n$ vertices can contain at most $3n - 8$ triangles, which aligns with the classical bound on maximal planar graphs.


\end{abstract}

\section{Introduction}
\label{sec:intro}

\subsection{Combinatorial Reconfiguration}

Recently, \textit{combinatorial reconfiguration} has emerged in different areas of computer science, including recreational mathematics (e.g., games and puzzles), computational geometry (e.g., flip graphs of triangulations), constraint satisfaction (e.g., solution space of Boolean formulas), and even quantum complexity theory (e.g., ground state connectivity).
Given a \textit{source problem} $\mathcal{P}$ (e.g., \textsc{Satisfiability}, \textsc{Vertex-Coloring},
\textsc{Independent Set}, etc.) and a prescribed \textit{reconfiguration rule} $\sfR$ that usually describes a ``small'' change in a feasible solution of $\mathcal{P}$ (e.g., satisfying truth assignments, proper vertex-colorings, independent sets, etc.) without affecting its feasibility (e.g., flipping one bit of a satisfying truth assignment, recoloring one vertex of a proper vertex-coloring, adding/removing one member of an independent set, etc.), one can define the corresponding \textit{reconfiguration graph of $\mathcal{P}$ under the rule $\sfR$} as follows.
In such a graph, each \textit{feasible solution} of $\mathcal{P}$ is a node, and two nodes are \textit{adjacent} if one can be obtained from the other by applying the rule $\sfR$ exactly once.
In particular, the reconfiguration graph of satisfying truth assignments of a Boolean formula on $n$ variables under the ``one-bit-flipping rule'', also known as the \textit{solution graph} of a Boolean formula, is an induced subgraph of the well-known \textit{hypercube} graph $Q_n$---the graph whose nodes are length-$n$ binary strings and two nodes are adjacent if they differ in exactly one bit. 
(See \cref{fig:exa-SAT-Reconf}.)

\begin{figure}[ht]
	\centering
	\includegraphics[width=0.3\textwidth]{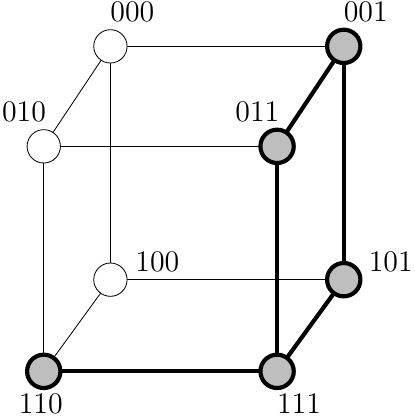}
	\caption{The reconfiguration graph of satisfying truth assignments of the Boolean formula $(x \land y) \lor z$ under the ``one-bit-flipping rule'' is an induced subgraph of $Q_3$. Each binary string corresponds to a truth assignment of the variables $x$, $y$, and $z$, respectively.}
	\label{fig:exa-SAT-Reconf}
\end{figure}

To the best of our knowledge, reconfiguration graphs have been studied from both \textit{algorithmic} and \textit{graph-theoretic} viewpoints.
From the \textit{algorithmic perspective}, the main goal is to understand whether certain algorithmic questions, such as finding a (shortest) path between two given nodes of the reconfiguration graph, can be answered efficiently, and if so, design an algorithm to do it.
A major challenge is that the size of a reconfiguration graph is often quite huge, and thus the whole graph can never be a part of the input.
From the \textit{graph-theoretic perspective}, the main goal is to study the structural properties (e.g., connectedness, bipartitedness, Hamiltonicity, etc.) of reconfiguration graphs as well as classify them based on known graph classes (which graphs are reconfiguration graphs). 
With respect to several well-known source problems, reconfiguration graphs have been extensively studied from the algorithmic viewpoint~\cite{Heuvel13,Nishimura18,BousquetMNS22}.
On the other hand, from the \textit{graph-theoretic perspective}, reconfiguration graphs have been well-characterized only for a limited number of source problems---namely those whose ``feasible solutions'' are satisfying truth assignments of a Boolean formula~\cite{Scharpfenecker15}, or general vertex subsets~\cite{MonroyFHHUW12}, (maximum) matchings~\cite{ErohS98}, dominating sets, or proper vertex-colorings~\cite{MynhardtN19} of a graph.
We refer readers to the surveys~\cite{Heuvel13,Nishimura18,MynhardtN19,BousquetMNS22} for more details on recent advances in this research area.

\subsection{Reconfiguration of Cliques}

In this paper, we consider \textsc{Clique} as the source problem. A \textit{clique} in a graph $G$ is a subset of vertices that are all adjacent. Each clique can be viewed as a set of tokens on the vertices of $G$, with no vertex holding more than one token. The following reconfiguration rules are well-known:
\begin{itemize}
	\item \textbf{Token Jumping ($\sfTJ$):} Two cliques are \textit{adjacent under $\sfTJ$} if one can be transformed into the other by moving a token to an unoccupied vertex.
	\item \textbf{Token Sliding ($\sfTS$):} Two cliques are \textit{adjacent under $\sfTS$} if one can be transformed into the other by moving a token to an adjacent unoccupied vertex.
	\item \textbf{Token Addition/Removal ($\sfTAR$):} Two cliques are \textit{adjacent under $\sfTAR_k$ (resp. $\sfTAR^k$)} if one can be transformed into the other by adding a token to an unoccupied vertex or removing one from an occupied vertex, ensuring the resulting set has at least (resp. at most) $k$ tokens. (Here, $\emptyset$ is considered a clique of size $0$.)
\end{itemize}
Given a graph $G$ and possibly an integer $k \geq 0$, one can define different reconfiguration graphs of cliques of $G$ as follows.
\begin{itemize}
	\item $\TJ{}{G}$ (resp. $\TJ{k}{G}$) is the graph whose nodes are cliques (resp. cliques of size $k$) of $G$ and edges are defined under $\sfTJ$. Similar definitions hold for $\TS{}{G}$ and $\TS{k}{G}$.
	\item $\TAR{k}{G}$ (resp. $\TARR{k}{G}$) is the graph whose nodes are cliques of size \textit{at least} (resp. \textit{at most}) $k$ of $G$ and edges are defined under $\sfTAR_k$ (resp. $\sfTAR^k$).  For convenience, we denote by $\TAR{}{G}$ the graph $\TAR{0}{G} \cong \TARR{\omega(G)}{G}$, where $\omega(G)$ is the maximum size of a clique of $G$. By definition, $\TAR{}{G}$ is indeed the union of $\TAR{k}{G}$ and $\TARR{k}{G}$.
\end{itemize}
Naturally, for $\sfR \in \{\sfTS, \sfTJ, \sfTAR\}$ and a given graph $G$, we call $G$ a \textit{$\sfR$-reconfiguration graph} (or simply \textit{$\sfR$-graph}) if there exists a graph $H$ such that $G \cong \R{}{H}$.
We define a \textit{$\sfR_k$-reconfiguration graph} (or simply \textit{$\sfR_k$-graph}) and a \textit{$\sfR^k$-reconfiguration graph} (or simply \textit{$\sfR^k$-graph}) similarly.

\begin{figure}[ht]
	\centering
	\includegraphics[width=0.7\textwidth]{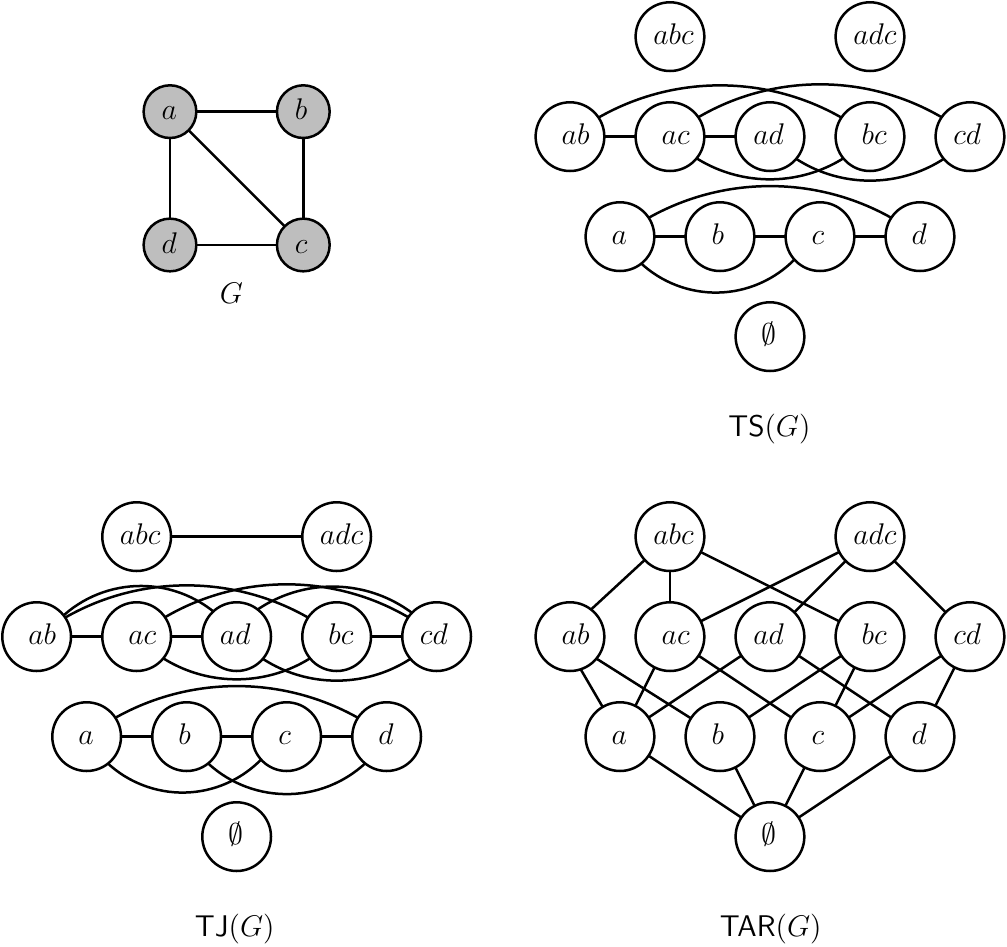}
	\caption{An example of a graph $G$ and some reconfiguration graphs of cliques of $G$. Each label inside a white-colored vertex (e.g., $abc$) indicates a clique of $G$ (e.g., $\{a, b, c\}$).}
\end{figure}

We review related results on reconfiguration graphs of cliques from both algorithmic and graph-theoretic perspectives.
Algorithmically, Ito~et~al.~\cite{ItoOO15,ItoOO23} initiated the study of \textsc{Clique Reconfiguration (CR)} under $\sfR \in \{\sfTS, \sfTJ, \sfTAR_k\}$, asking whether a path exists between two given cliques in $\sfR(G)$.
They showed that $\sfTS$, $\sfTJ$, and $\sfTAR_k$ are equivalent regarding polynomial-time solvability, with \textsc{CR} being PSPACE-complete for perfect graphs but polynomial-time solvable for even-hole-free graphs and cographs.
They also designed polynomial-time algorithms for the shortest path variant when $G$ is chordal, bipartite, planar, or has bounded treewidth.

A different way of looking at a clique is to consider it as a collection of edges instead of vertices.
With this viewpoint, the \textit{edge-variants} of \textsc{CR} can be defined similarly under all above rules.
Hanaka~et~al.~\cite{HanakaIMMNSSV20} initiated the study of these edge-variants (as restricted cases of a more generalized problem called \textsc{Subgraph Reconfiguration}) from the algorithmic viewpoint by showing that the problems can be solved in linear time under any of $\sfTS$ and $\sfTJ$.

From the graph-theoretic viewpoint, the $\sfTAR$-graph (of a graph) was first introduced in 1989 by Bandelt and {van de Vel}~\cite{BandeltV89} under the name \textit{simplex graph}\footnote{Apparently, in some contexts, a \textit{simplex} is a vertex subset where any two members form an edge (which is the same as our definition of a clique) and a \textit{clique} is a maximal simplex which is not contained in any larger simplex.}.
The $\sfTS_k$-graph of a complete graph $K_n$ is closely related to the so-called \textit{token graphs}~\cite{MonroyFHHUW12} and \textit{Johnson graphs}~\cite{HoltonS93}.
A \textit{token graph} of a graph $G$, denoted by $F_k(G)$, is a graph whose vertices are size-$k$ vertex-subsets of $G$ and two vertices are adjacent if one can be obtained from the other by applying a single $\sfTS$-move.
A \textit{Johnson graph} $J(n, k)$ is a graph whose vertices are size-$k$ subsets of an $n$-element set and two vertices are adjacent if their intersection is of size exactly $k-1$.
One can readily verify that $\sfTS_k(K_n) \cong F_k(K_n) \cong J(n, k)$.
To the best of our knowledge, reconfiguration graphs of cliques under $\sfTS$ or $\sfTJ$ have not yet been systematically studied in the literature.

\subsection{Our Problem and Results}

In this paper, we consider reconfiguration graphs of cliques and study their structural properties.
We prove a number of interesting results, most of which are under $\sfTS$ and $\sfTJ$ rules. 

In \cref{sec:TS}, we study the maximum clique size and chromatic number of $\sfTS_k(G)$. We establish a formula relating $\omega(G)$ and $\omega(\sfTS_k(G))$, and bound $\chi(\sfTS_k(G))$ via the chromatic number of an appropriate Johnson graph.

In \cref{sec:TJ}, we focus on $\sfTJ_k$-graphs with $k = \omega(G)$. We give an algorithm to construct $\sfTS_{\omega(G)-1}(G)$ from $\sfTJ_{\omega(G)}(G)$ and derive structural properties of maximum $\sfTJ$ graphs, i.e., graphs $H$ for which $H \cong \sfTJ_{\omega(G)}(G)$ for some $G$. 

Finally, in \cref{sec:planar-TSk}, we prove that $\sfTS_k(G)$ is planar whenever $G$ is planar. We also derive bounds on the number of $3$- and $4$-cliques in terms of $|E(G)|$. Specifically, for $3$-cliques, we prove that any planar graph $G$ with $n = |V(G)|$ vertices can have at most $3n - 8$ copies of $K_3$, which matches the classical bound on maximal planar graphs~\cite{hakimi1979number}.

\section{Preliminaries}
\label{sec:preliminaries}

For the concepts and notations not defined here, we refer readers to~\cite{Diestel2017}.
Unless otherwise mentioned, all graphs in this paper are simple, connected, and undirected.
We denote by $V(G)$ and $E(G)$ the vertex-set and edge-set of a graph $G$, respectively.
The \textit{chromatic number} of $G$, denoted by $\chi(G)$, is the smallest number of colors that can be used to color vertices of $G$ such that no two adjacent vertices receive the same color.
For two sets $X, Y$, we sometimes write $X + Y$ and $X - Y$ to respectively indicate $X \cup Y$ and $X \setminus Y$.
When $Y = \{y\}$, we simply write $X + y$ and $X - y$ instead of $X + \{y\}$ and $X - \{y\}$, respectively.
We denote by $X \Delta Y$ the \textit{symmetric difference} of $X$ and $Y$, i.e., $X \Delta Y = (X - Y) + (Y - X)$.
For a vertex-subset $X$, we denote by $G[X]$ the subgraph of $G$ induced by vertices in $X$.

We now formally define the $\sfTS$, $\sfTJ$, and $\sfTAR$ rules.
Two cliques $C$ and $C^\prime$ of $G$ are \textit{adjacent under $\sfTJ$} if there exist $u, v \in V(G)$ such that $C - C^\prime = \{u\}$ and $C^\prime - C = \{v\}$.  
We say that they are \textit{adjacent under $\sfTS$} if one additional constraint $uv \in E(G)$ is satisfied.
Two cliques $C$ and $C^\prime$ are \textit{adjacent under $\sfTAR_k$ (resp. $\sfTAR^k$)} if $|C \Delta C^\prime| = 1$ and $\min\{|C|, |C^\prime|\} \geq k$ (resp. $\max\{|C|, |C^\prime|\} \leq k$).

\section{Token Sliding}
\label{sec:TS}

We prove a useful observation regarding cliques in $\sfTS_k$-graphs.
We remark that similar observations for Johnson graphs have been proved in~\cite{ShuldinerO22}.
\begin{lemma}\label{lem:clique-in-TSgraph}
	Let $G$ be a graph.
	Suppose that $\sfTS_k(G)$ contains a complete subgraph $K_n$ ($n \geq 3$) whose vertices $A_1, \dots, A_n$ are $k$-cliques of $G$.
	There exist either a clique $\Uni \subseteq V(G)$ of size $k + 1$ and pairwise distinct vertices $a_1, \dots, a_n \in \Uni$ such that $A_i = \Uni - a_i$ or a clique $\Int \subseteq V(G)$ of size $k - 1$ and pairwise distinct vertices $a_1, \dots, a_n \in V(G) \setminus \Int$ such that $A_i = \Int + a_i$, where $1 \leq i \leq n$.
	In particular, when $n > k + 1$, a clique $\Int$ satisfying the above conditions exists.
\end{lemma}
\begin{proof}
	Observe that $|A_i \cap A_j| = k - 1$ for $1 \leq i, j \leq n$.
	We prove the statement by induction on $n$.
	For the base case $n = 3$, let $A = A_1 \cap A_2$ and $B = A_1 \cap A_3$.
	By definition, $|A| = |B| = k-1$.
	As $A$ and $B$ are both subsets of $A_1$ of size $k-1$, it follows that $|A_1| = k \geq |A \cup B| = |A| + |B| - |A \cap B| = 2k - 2 - |A \cap B|$.
	Thus, $|A \cap B| \geq k - 2$. 
	On the other hand, $|A \cap B| \leq |A| = |B| = k-1$.
	Let $S = A \cap B = A_1 \cap A_2 \cap A_3$.
	We have two cases: $|S| = k-1$ and $|S| = k-2$.
	\begin{itemize}
		\item If $|S| = k-1$, it follows that there exist pairwise distinct vertices $a, b, c \in V(G)$ such that $A_1 = S + a$, $A_2 = S + b$, and $A_3 = S + c$.
		In this case, we can define $\Int = S$.
		\item If $|S| = k-2$, it follows that there exist pairwise distinct vertices $a, b, c \in V(G)$ such that $A_1 = S + a + b$, $A_2 = S + b + c$, and $A_3 = S + a + c$.
		In this case, we can define $\Uni = S + a + b + c$.
	\end{itemize}

	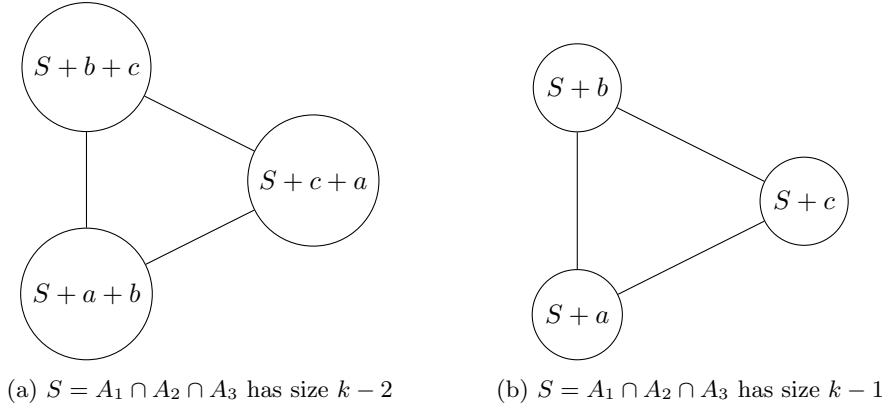
\begin{figure}[ht]
		\centering
		\begin{subfigure}[b]{0.4\textwidth}
		  \centering
		  \begin{tikzpicture}[every node/.style = {circle, draw, fill=white, minimum size=1cm}]
			\foreach \l/\x/\y/\c in {Sab/0/0/{S+a+b}, Sbc/0/3/{S+b+c}, Sca/3/1.5/{S+c+a}}
			{
			  \node (\l) at (\x, \y) {$\c$};
			}
			\draw (Sab) -- (Sbc) -- (Sca) -- (Sab);
		  \end{tikzpicture}
		  \caption{$S = A_1 \cap A_2 \cap A_3$ has size $k - 2$}
		\end{subfigure}  
		\begin{subfigure}[b]{0.4\textwidth}
		  \centering
		  \begin{tikzpicture}[every node/.style = {circle, draw, fill=white, minimum size=1cm}]
			\foreach \l/\x/\y/\c in {Sa/0/0/{S+a}, Sb/0/3/{S+b}, Sc/3/1.5/{S+c}}
			{
			  \node (\l) at (\x, \y) {$\c$};
			}
			\draw (Sa) -- (Sb) -- (Sc) -- (Sa);
		  \end{tikzpicture}
		  \caption{$S = A_1 \cap A_2 \cap A_3$ has size $k - 1$}
		\end{subfigure}  
		\caption{Illustration for the base case $n = 3$}
	  \end{figure}

	Suppose that the lemma holds for $n - 1$.
	We claim that it also holds for $n$.
	Suppose that the $k$-cliques $A_1, \dots, A_n$ form a $K_n$ in $\sfTS_k(G)$.
	Thus, $A_1, \dots, A_{n-1}$ clearly form a $K_{n-1}$ in $\sfTS_k(G)$.
	By the inductive hypothesis, for $1 \leq i \leq n-1$, there exist either
	\begin{itemize}
		\item[(i)] a clique $\Uni_{n-1} \subseteq V(G)$ of size $k + 1$ and pairwise distinct vertices $a_1, \dots, a_{n-1} \in \Uni_{n-1}$ such that $A_i = \Uni_{n-1} - a_i$; or
		\item[(ii)] a clique $\Int_{n-1} \subseteq V(G)$ of size $k - 1$ and pairwise distinct vertices $a_1, \dots, a_{n-1} \in V(G) \setminus \Int_{n-1}$ such that $A_i = \Int_{n-1} + a_i$.
	\end{itemize}
	
	From our assumption, note that $A_1$, $A_2$, and $A_n$ form a $K_3$ in $\sfTS_k(G)$.
	From the base case, either $A_n = (A_1 \cap A_2) + x$ (when $|A_1 \cap A_2 \cap A_n| = k - 1$) or $A_n = (A_1 \cup A_2) - x$ (when $|A_1 \cap A_2 \cap A_n| = k - 2$), for some $x \in V(G)$.
	We consider four cases:
	\begin{itemize}
		\item \textbf{When (i) holds and $A_n = (A_1 \cap A_2) + x$.} 
		In this case, since $|A_1 \cap A_2| = k-1 = |A_n| - 1$, we have $x \notin A_1 \cap A_2$.
		Additionally, since $A_1 \neq A_n$, it follows that $x \notin A_1$. Similarly, $x \notin A_2$.
		Therefore, $x \notin A_1 \cup A_2$.
		Since $a_2 \in A_1$ and $a_1 \in A_2$, we have $A_1 \cup A_2 = \Uni_{n-1}$. 
		(Indeed, one can verify that $A_i \cup A_j = \Uni_{n-1}$ for any pair of pairwise distinct $i, j \in \{1, \dots, n-1\}$.)
		Since $x \notin A_1 \cup A_2 = \Uni_{n-1}$ and $A_3 = \Uni_{n-1} - a_3 \subseteq \Uni_{n-1}$, we have $x \notin A_3$.
		Additionally, since $x \notin \Uni_{n-1}$ and $a_3 \in \Uni_{n-1}$, we have $x \neq a_3$.
		Therefore, $\{x, a_3\} \subseteq A_n \setminus A_3$, which implies $|A_n \setminus A_3| \geq 2$.
		This contradicts the adjacency of $A_n$ and $A_3$ in $\sfTS_k(G)$.
		\item \textbf{When (i) holds and $A_n = (A_1 \cup A_2) - x$.} 
		Again, since $a_2 \in A_1$ and $a_1 \in A_2$, we have $A_1 \cup A_2 = \Uni_{n-1}$. 
		Thus, $x \in \Uni_{n-1}$.
		Moreover, as $A_n \notin \{A_1, \dots, A_{n-1}\}$, it follows that $x \notin \{a_1, \dots, a_{n-1}\}$.
		Hence, we set $\Uni = \Uni_{n-1}$ and $a_n = x$.
		The lemma holds for $n$.
		\item \textbf{When (ii) holds and $A_n = (A_1 \cap A_2) + x$.}
		In this case, $A_1 \cap A_2 = \Int_{n-1}$.
		(Indeed, one can verify that $A_1 \cap A_2 \cap \dots \cap A_{n-1} = \Int_{n-1}$.)
		Since $|A_1 \cap A_2| = k - 1 = |A_n| - 1$, we have $x \notin A_1 \cap A_2 = \Int_{n-1}$.
		Additionally, since $A_n \notin \{A_1, \dots, A_{n-1}\}$, it follows that $x \notin \{a_1, \dots, a_{n-1}\}$.
		Hence, we set $\Int = \Int_{n-1}$ and $a_n = x$.
		The lemma holds for $n$.
		\item \textbf{When (ii) holds and $A_n = (A_1 \cup A_2) - x$.}
		In this case, we have $A_n = (\Int_{n-1} + a_1 + a_2) - x$.
		Since $A_n \neq A_1$, we have $x \neq a_2$.
		Similarly, $x \neq a_1$.
		Since $|\Int_{n-1} + a_1 + a_2| = k + 1 > |A_n| = k$, it follows that $x \in \Int_{n-1}$.
		Since $x \in \Int_{n-1}$ and $a_3 \notin \Int_{n-1}$, we have $x \neq a_3$. 
		Additionally, $x \in A_3 = \Int_{n-1} + a_3$.
		Moreover, since $A_3 \notin \{A_1, A_2\}$, it follows that $a_3 \notin \{a_1, a_2\}$, and therefore $a_3 \notin \Int_{n-1} + a_1 + a_2$.
		Hence, $a_3 \notin A_n$.
		Therefore, $\{x, a_3\} \subseteq A_3 \setminus A_n$, which implies $|A_3 \setminus A_n| \geq 2$. 
		This contradicts the adjacency of $A_n$ and $A_3$ in $\sfTS_k(G)$.
	\end{itemize}
    In each of the above cases, we either verify the lemma directly, or derive a contradiction, which shows that this case cannot occur. Therefore, the lemma holds for all $n$. Our inductive proof is complete.

	Finally, we show that when $n > k + 1$, a clique $\Int$ satisfying the lemma exists.
	Note that a clique $\Uni$, if exists, must be of size exactly $k+1$ and contains at least $n$ members $a_1, \dots, a_n$.
	Clearly, this happens only when $n \leq k + 1$.
	Therefore, when $n > k + 1$, no clique $\Uni$ exists, and therefore $\Int$ exists.
\end{proof}

We now prove a relationship between the clique number of $G$ and that of $\sfTS_k(G)$.
A similar relationship has been established for Johnson graphs in~\cite{ShuldinerO22}.
\begin{theorem}\label{thm:clique-num-TSgraph}
	Let $G$ be a graph.
	\begin{enumerate}[(a)]
		\item If $k > \omega(G)$ then $\omega(\sfTS_k(G)) = 0$.
		\item If $k = \omega(G)$ then $\omega(\sfTS_k(G)) = 1$.
		\item If $k < \omega(G)$ then $\omega(\sfTS_k(G)) = \max\{k + 1, \omega(G) - k + 1\}$.
	\end{enumerate}
\end{theorem}
\begin{proof}
	\begin{enumerate}[(a)]
		\item Trivial.
		\item We show that if $k = \omega(G)$ then $\sfTS_k(G)$ has no edge, which implies $\omega(\sfTS_k(G)) = 1$.
		Suppose to the contrary that $IJ \in E(\sfTS_k(G))$.
		Thus, $I \cup J$ forms a clique of size $k + 1$ in $G$, which contradicts the assumption $k = \omega(G)$.
		\item We first show that $\omega(\sfTS_k(G)) \geq \max\{k+1, \omega(G) - k + 1\}$.
		To this end, we describe how to construct a $(k+1)$-clique and a $(\omega(G) - k + 1)$-clique in $\sfTS_k(G)$, respectively.
		\begin{itemize}
			\item To construct a clique of size $k + 1$ in $\sfTS_k(G)$, let $A = \{a_1, \dots, a_{k+1}\}$ be any clique of size exactly $k + 1$ in $G$.
			Since $\omega(G) > k$, such a clique $A$ exists.
			Clearly, the vertices $A - a_i$ ($1 \leq i \leq k+1$) form a $(k+1)$-clique of $\sfTS_k(G)$.

			\item To construct a clique of size $\omega(G) - k + 1$ in $\sfTS_k(G)$, let $B = \{b_1, \dots, b_{\omega(G)}\}$ be any clique of size exactly $\omega(G)$ in $G$ and let $C = \{b_1, \dots, b_{k-1}\} \subseteq B$.
			Clearly, the vertices $C + b_j$ ($k \leq j \leq \omega(G)$) form a $(\omega(G) - k + 1)$-clique of $\sfTS_k(G)$.
		\end{itemize}
		
		It remains to show that $\omega(\sfTS_k(G)) \leq \max\{k+1, \omega(G) - k + 1\}$.
		To this end, we show that if $m = \omega(\sfTS_k(G)) > k + 1$ then $m \leq \omega(G) - k + 1$.
		Equivalently, we show that if $m = \omega(\sfTS_k(G)) > k + 1$ then $G$ has a clique of size $m + k - 1$ (which then implies $\omega(G) \geq m + k - 1$ and thus $m \leq \omega(G) - k + 1$).
		Let $C$ be a $m$-clique of $\sfTS_k(G)$ whose vertices $A_1, \dots, A_m$ are $k$-cliques of $G$.
		By \cref{lem:clique-in-TSgraph}, there exist a clique $\Int$ of size $k - 1$ and pairwise distinct vertices $a_1, \dots, a_m \in V(G) \setminus \Int$ such that $A_i = \Int + a_i$, where $1 \leq i \leq m$.
		As $A_i$ and $A_j$ are adjacent for any pairwise distinct pair $i, j \in \{1, \dots, m\}$, it follows that $a_ia_j \in E(G)$.
		Thus, $\Int + \{a_1, a_2, \dots, a_m\}$ is a $(m + k - 1)$-clique of $G$.
		Our proof is complete.
	\end{enumerate}
\end{proof}

In the next propositions, we will link the chromatic number of $\sfTS_k(G)$ with that of a Johnson graph. 
For a set $X$, we denote by $P(X)$ the \textit{power set} of $X$, i.e., the set of all subsets of $X$.
\begin{proposition}
    Given an arbitrary graph $G$ and number $k$, $\chi (\sfTS_k(G)) \leq \chi (J(\chi (G),k))$.
\end{proposition}
\begin{proof}
    Let $n=\chi(G)$ and $m=\chi(J(n,k))$.
	Let $c_1: V(G) \to \{1,2, \dots, n\}$ be a proper coloring of $G$ and $c_2: V(J(n,k)) \to \{1,2, \dots, m\}$ be a proper coloring of $J(n,k)$.

	We define $f: V(\sfTS_k(G)) \to P(\{1, 2, \dots, n\})$, $$f(A)=\{x\in \{1,2, \dots n\}: \exists y \in A,x=c_1(y) \}.$$ 
	We will prove that $f(A)$ is a $k$-element subset of $\{1,2, \dots,n\}$ and for any $A,B\in V(\sfTS_k(G))$ if $AB$ is an edge of $\sfTS_k(G)$ then $f(A)f(B)$ is an edge of $J(n,k)$.
     
Firstly, we will prove for all $A\in V(\sfTS_k(G))$, the set $f(A)$ has exactly $k$ elements. 
Let $A=\{a_1,a_2, \dots,a_k\}$ be any cliques with $k$ elements in $G$, we know that for all $i \neq j$, $a_ia_j\in E(G)$ so $c_1(a_i)\neq c_1(a_j)$. 
Hence, $f(A)=\{c_1(a_1),c_1(a_2), \dots, c_1(a_k)\}$, which is a $k$-element subset of $\{1,2, \dots,n\}$.

Secondly, we will prove that for all $A,B \in V_{\sfTS_k(G)}$ such that $AB\in E(\sfTS_k(G))$, we have $f(A)f(B)\in E(J(n,k))$. 
Let $A=\{a_1,a_2,\dots,a_{k-1},a_k\}$ and $B=\{a_1,a_2,\dots,a_{k-1},a_{k+1}\}$ be two cliques in $G$ such that $AB$ is an edge in $TS_k(G)$($a_ka_{k+1}$ has to also be an edge in $G$), we know that $f(A)=\{c_1(a_1),c_1(a_2),\dots,c_1(a_{k-1}),c_1(a_k)\}$, $f(B)=\{c_1(a_1),c_1(a_2),\dots,c_1(a_{k-1}),c_1(a_{k+1})\}$. 
Hence, $f(A)\cap f(B)=\{c_1(a_1),c_1(a_2),\dots,c_1(a_{k-1})\}$ because $c_1(a_k)\neq c_1(a_{k+1})$ ($a_ka_{k+1}$ is also an edge in $G$).

After that, we consider a coloring for $\sfTS_k(G)$, $c_3:V_{\sfTS_k(G)}\rightarrow \{1,2,\dots,m\}$ $$c_3(A)=c_2(f(A)).$$ 
Now let $AB$ be an edge in $\sfTS_k(G)$. 
Then, $f(A)f(B)$ is an edge in $J(n,k)$, which means $c_3(A)=c_2(f(A))\neq c_2(f(B))=c_3(B)$, so $c_3$ is a valid coloring with $m$ colors.
\end{proof}

\begin{proposition}
    Given an arbitrary graph $G$ and number $k$, $\chi (\sfTS_k(G)) \geq \chi (J(\omega (G),k))$.
\end{proposition}
\begin{proof}
    Let $H$ be a clique of size $\omega(G)$ in $G$. 
	We know that $\sfTS_k(H)$ is a subgraph of $TS_k(G)$ and $J(\omega (G),k) \cong TS_k(H)$. Hence $\chi (\sfTS_k(G))\geq \chi (J(\omega (G),k))$.
\end{proof}

\section{Token Jumping}
\label{sec:TJ}

In this section, we consider $\sfTJ_k$-graphs.
As $\sfTS_k(G)$ is a subgraph of $\sfTJ_k(G)$, we immediately have the following consequence of \cref{thm:clique-num-TSgraph}.

\begin{corollary}
\label{cor:maxclique_TJ}
	If $k < \omega(G)$ then $\omega(\sfTJ_k(G)) \geq \max\{k + 1, \omega(G) - k + 1\}$.
\end{corollary}

For a graph $G$, we now consider $\sfTJ_{\omega(G)}(G)$ and prove some of its properties.

\begin{lemma}\label{lem:max-clique-intersect-TJ}
	Suppose that the $\omega(G)$-cliques $A, B, C$ of a graph $G$ form a triangle in $\sfTJ_{\omega(G)}(G)$.
	Then, $A \cap B = B \cap C = A \cap C$.
\end{lemma}
\begin{proof}
	We remark that the proof of \cref{lem:clique-in-TSgraph} for the base case $n = 3$ uses only basic properties from the set theory and the assumption that for any pair of adjacent vertices $A, B$ in $\sfTS_k(G)$, $|A \cap B| = k-1$.
	Thus, the same argument in \cref{lem:clique-in-TSgraph} can be applied for vertices $A$, $B$, $C$ of $\sfTJ_k(G)$ ($1 \leq k \leq \omega(G)$) to show that if they form a $K_3$ in $\sfTJ_k(G)$ then $A$, $B$, and $C$ have one of the two forms: either $A = S + a$, $B = S + b$, and $C = S + c$; or $A = S + a + b$, $B = S + b + c$, and $C = S + a + c$, for some clique $S$ of $G$ and pairwise distinct vertices $a, b, c$ not in $S$.
	In the former case, $S = A \cap B \cap C = A \cap B = B \cap C = A \cap C$.
	In the latter case, one can verify that the vertices $a, b, c$ form a triangle in $G$, and therefore $S + a + b + c$ is a clique in $G$ of size $|A| + 1$. When $|A| = \omega(G)$, this is a contradiction.
	Thus, our proof is complete.
\end{proof}

Using \cref{lem:max-clique-intersect-TJ}, we can prove the following proposition.
A \textit{diamond} $K_4 - e$ is a graph obtained from $K_4$ by removing exactly one edge.
\begin{proposition}
\label{prop:maximal_TJ_is_diamond_free}
	For any graph $G$, the graph $\sfTJ_{\omega(G)}(G)$ does not contain any diamond as an induced subgraph.
\end{proposition}
\begin{proof}
	Suppose to the contrary that $\sfTJ_{\omega(G)}(G)$ contains a diamond graph $H$ as an induced subgraph.
	Let $A, B, C, D$ be vertices of $H$ and assume w.l.o.g that $A$ and $D$ are not adjacent.
	From \cref{lem:max-clique-intersect-TJ}, since $A, B, C$ form a triangle in $\sfTJ_{\omega(G)}(G)$, we have $A \cap B = B \cap C = A \cap C$.
	Similarly, we have $B \cap C = C \cap D = B \cap D$.
	The former equation gives us $B \cap C \subseteq A$ and the latter $B \cap C \subseteq D$.
	Thus, $B \cap C \subseteq A \cap D$ and therefore $|B \cap C| \leq |A \cap D|$.
	On the other hand, since $B, C$ are adjacent and $A, D$ are not, we have $|B \cap C| = \omega(G) - 1 > |A \cap D|$, a contradiction.
\end{proof}

\begin{proposition}\label{prop:Siu-TJwG}
	Let $G$ be any graph and let $k = \omega(G)$.
	Let $U$ be a vertex of $\sfTJ_{k}(G)$.
	One can partition the set $N_{\sfTJ_{k}(G)}(U)$ of neighbors of $U$ in $\sfTJ_{k}(G)$ into $k$ disjoint (possibly empty) subsets $S_1(U), S_2(U), \dots, S_k(U)$ such that each $S_i(U)$ ($1 \leq i \leq k$), if it is not empty, is a clique of $\sfTJ_k(G)$.
	Moreover, there is no edge in $\sfTJ_k(G)$ joining a vertex of $S_i(U)$ and that of $S_j(U)$, for distinct pair $i, j \in \{1, \dots, k\}$.
\end{proposition}
\begin{proof}
	For two neighbors $V, W$ of $U$ in $\sfTJ_k(G)$, we show that if $U \cap V = U \cap W$ then $V$ and $W$ are adjacent in $\sfTJ_k(G)$.
	In this case, since $U \cap V$ is a subset of both $V$ and $W$, we have $U \cap V \subseteq V \cap W$.
	Hence, $|V \cap W| \geq |U \cap V| = k - 1$.
	Moreover, $|V \cap W| < |V| = |W| = k$.
	Therefore, $|V \cap W| = k - 1$, which implies that $V$ and $W$ are adjacent in $\sfTJ_k(G)$.
	
	Initially, for each $i \in \{1, \dots, k\}$, we set $S_i(U) = \emptyset$.
	Let $A_1, A_2, \dots, A_k$ be all size-$(k-1)$ subsets of $U$.
	For each $V \in N_{\sfTJ_{k}(G)}(U)$, if $U \cap V = A_i$ for some $i \in \{1, \dots, k\}$, we add $V$ to $S_i(U)$.
	Using the above claim, one can verify that if $S_i(U)$ ($1 \leq i \leq k$) is non-empty then any pair $V, W \in S_i(U)$ are adjacent and therefore $S_i(U)$ is a clique of $\sfTJ_k(G)$.
	
	Finally, if $V \in S_i(U)$ and $W \in S_j(U)$ for distinct pair $i, j \in \{1, \dots, k\}$, we have $U \cap V \neq U \cap W$.
	By \cref{lem:max-clique-intersect-TJ}, $U, V, W$ does not form a triangle in $\sfTJ_k(G)$, and thus there is no edge joining $V$ and $W$.
\end{proof}

Next, we prove an interesting relationship between the two graphs $\sfTJ_{\omega(G)}(G)$ and $\sfTS_{\omega(G) - 1}(G)$, for any graph $G$.
The following lemma is useful.

\begin{lemma}\label{lem:TSk-1-vs-TJk}
	Let $G$ be a graph and $k$ be an integer satisfying $2 \leq k \leq \omega(G)$.
	Let $A, B \in V(\sfTS_{k-1}(G))$.
	Then, $AB \in E(\sfTS_{k-1}(G))$ if and only if $A \cup B \in V(\sfTJ_{k}(G))$.
\end{lemma}
\begin{proof}
	\begin{itemize}
		\item[($\Rightarrow$)] Suppose that $AB \in E(\sfTS_{k-1}(G))$.
		We show that $A \cup B \in V(\sfTJ_{k}(G))$.
		Since $A, B \in V(\sfTS_{k-1}(G))$, we have $|A| = |B| = k - 1$ and $|A \cap B| = k - 2$.
		Thus, $|A \cup B| = |A| + |B| - |A \cap B| = k$.
		To show that $A \cup B \in V(\sfTJ_{k}(G))$, it remains to verify that $A \cup B$ is a clique of $G$.
		Since $AB \in E(\sfTS_{k-1}(G))$, it follows that there exist $u, v \in V(G)$ such that $A \setminus B = \{u\}$, $B \setminus A = \{v\}$, and $uv \in E(G)$.
		Since $u \in A$, it is adjacent to every vertex of $A \cap B$.
		Similarly, so is $v \in B$.
		Moreover, $u$ and $v$ are adjacent.
		Thus, $A \cup B = (A \cap B) + u + v$ is indeed a clique of $G$.

		\item[($\Leftarrow$)] Suppose that $A \cup B \in V(\sfTJ_k(G))$.
		We show that $AB \in E(\sfTS_{k-1}(G))$, that is, there exist $u, v \in V(G)$ such that $A \setminus B = \{u\}$, $B \setminus A = \{v\}$, and $uv \in E(G)$.
		Observe that $|A \cap B| =  |A| + |B| - |A \cup B| = k - 2$.
		Additionally, since $|A| = |B| = k - 1$, it follows that $|A \setminus B| = |B \setminus A| = |A| - |A \cap B| = |B| - |A \cap B| = 1$.
		As a result, there exist $u, v \in V(G)$ such that $A \setminus B = \{u\}$ and $B \setminus A = \{v\}$.
		Since $u, v \in A \cup B$ and $A \cup B \in V(\sfTJ_k(G))$, we have $uv \in E(G)$.
	\end{itemize}
\end{proof}

We remark that for two graphs $G$ and $T$ satisfying $T \cong \sfTJ_k(G)$, one can indeed label vertices of $T$ by $k$-cliques of $G$.
We call such a labelling a \textit{set label} of $T$.
Given a graph $T \cong \sfTJ_k(G)$ and one of its set labels, using \cref{lem:TSk-1-vs-TJk}, one can indeed construct a graph $H$ such that $\sfTS_{k-1}(G) \cong H + cK_1$ for some integers $k \geq 2$ and $c \geq 0$: vertices of $H$ are exactly the size-$(k-1)$ subsets of each $k$-clique of $G$ (i.e., vertex in $\sfTJ_k(G)$), and two vertices of $H$ are adjacent if their union is a $k$-clique of $G$ (\cref{lem:TSk-1-vs-TJk}).


We now consider a more generalized case when $T$ is given without any set label.
Now for an integer $k$ and graph $G$, we call $G$ a \textit{$k$-good} graph if it satisfies the following \textit{$k$-good conditions}: for every vertex $u$ of $G$, the set $N_G(u)$ can be partitioned into $k$ disjoint subsets (possibly empty) $S_1(u), S_2(u), \dots, S_k(u)$ such that each $S_i(u)$ ($1 \leq i \leq k$), if it is not empty, is a clique in $G$ and there is no edge joining a vertex of $S_i(u)$ and that of $S_j(u)$, for a distinct pair $i, j \in \{1, \dots, k\}$. 

\begin{proposition}\label{prop:k-good-conds}
For a $k$-good graph $T = (V, E)$, one can find a multiset $\mathbf{Msets}$ of subsets of $V$ such that the following conditions are satisfied.
    \begin{enumerate}[(a)]
        \item For each pair of different sets $X,Y\in \mathbf{Msets}$, $|X\cap Y|<2$.
        \item For each $X\in \mathbf{Msets}$, $X$ is a clique in $T$.
        \item For each $u\in V(T)$, there are exactly $k$ sets $X\in \mathbf{Msets}$ such that $u\in X$.
        \item For each $uv\in E(T)$, there exists exactly one set $X\in \mathbf{Msets}$ such that $u,v\in X$.
    \end{enumerate}
\end{proposition}
\begin{proof}


We now describe how to construct $\mathbf{Msets}$. 
Initially, $\mathbf{Msets} = \emptyset$.
For each $u \in V(T)$, as $T$ is $k$-good, there exist $k$ sets $S_1(u),S_2(u),\dots,S_k(u)$ satisfying the $k$-good conditions.
For each $u \in V(T)$, we define $M_i(u) = S_i(u) \cup \{u\}$.
Now, if either $|M_i(u)| = 1$ or $M_i(u) \notin \mathbf{Msets}$, we add $M_i(u)$ to $\mathbf{Msets}$.
From this construction, note that if $M_i(u)$ is of size $1$ then it has multiplicity at least $1$ in $\mathbf{Msets}$, and otherwise it has multiplicity exactly $1$.

Additionally, observe that for any $u, v \in V(T)$ and any pair $i, j \in \{1, \dots, k\}$, if $|M_i(u) \cap M_j(v)| \geq 2$ then $M_i(u) = M_j(v)$.
To see this, suppose that there exist two distinct vertices $w, r \in M_i(u) \cap M_j(v)$.
Because edge $wr$ is in $T$, we assume that $r\in S_1(w)$ and $w\in S_1(r)$. We have  $M_1(w)=M_1(r)$ because they are the set of vertices that are incident with both of $\{w,r\}$ union with set $\{w,r\}$.

One of $w,r$ is different from $u$ we assume $w \neq u$. All the vertices connected to $w,r$ are inside $M_1(w)$ and hence $M_i(u) \subset M_1(w)$. All the vertex connect to $u,w$ are inside $M_i(u)$ and hence $M_1(w) \subset M_i(u)$. Then we have $M_1(w)=M_i(u)$ and similarly, $M_1(w)=M_j(v)$ or $M_1(r)=M_j(v)$. Then, either way, make $M_i(u)=M_j(v)$.

We are now ready to prove that the constructed set $\mathbf{Msets}$ satisfies the conditions (a)--(d).
\begin{enumerate}[(a)]
    \item For each different set $X,Y\in \mathbf{Msets}$. If $X\cap Y \geq 2$ then $X=Y$ and $|X|\geq 2$ however, every duplicate copy of sets with at least 2 elements has been eliminated, so this cannot happen.
    \item For every $X\in \mathbf{Msets}$, there exists $u\in V(T)$ and $1\leq i\leq k$ such that $X=M_i(u)=S_i(u)+u$. Because $S_i(u)$ is a clique and every element in it is a neighbor of $u$, $M_i(u)$ also forms a clique.
    \item For each $u\in V(T)$, the sets $M_1(u), M_2(u), \dots, M_k(u)$ contain $u$. 
	Since for the sets $M_i(u)$ for $1 \le i \le k$, the number of sets equals $\{u\}$ will still remain in $\mathbf{Msets}$ and the sets with at least 2 elements are pair-wise different, so $\mathbf{Msets}$ have to contain each copy of them. 
	As a result, there are at least $k$ sets in $\mathbf{Msets}$ that contain $u$. 
	Now, for each $X\in \mathbf{Msets}$ that contains $u$, we prove that $X$ can be a copy of $M_i(u)$ for some $1\leq i \leq k$. 
	\begin{itemize}
        \item  If $X=\{u\}$, then $X$ can only be obtained by $M_i(u)$.
        \item Otherwise, $X$ has at least 2 elements. Let $X=M_j(v)$. Here we consider the case $v\neq u$, then $v,u$ are incident in $T$, which means there exists $S_i(u)$ contain $v$. This indicates that $\{u,v\}\subset [M_i(u)\cap M_j(v)]$ which means that $M_i(u)=M_j(v)$ and hence $X$ is a copy of $M_i(u)$.
	\end{itemize}
    Hence, for all $u\in V(T)$, there are exactly $k$ sets $X\in U$ that contain $u$.
    \item For each $uv\in E(T)$, there exist $i$ such that $v\in S_i(u)$; hence, $M_i(u)$ contains both $u,v$, which means there is $X\in \mathbf{Msets}$ such that $u,v\in X$.
\end{enumerate} 
\end{proof}

From \cref{prop:Siu-TJwG}, $\sfTJ_{\omega(G)}$ is $\omega(G)$-good.
Therefore, the following corollary directly follows from \cref{prop:k-good-conds}.
\begin{corollary}
	Let $G$ be any graph and let $k = \omega(G)$.
    There exists a multiset $\mathbf{Msets}(\sfTJ_k(G))$ containing subsets of $V(\sfTJ_k(G))$ such that the following conditions are satisfied.
    \begin{enumerate}[(a)]
        \item For each different sets $X,Y\in \mathbf{Msets}(\sfTJ_k(G))$, $|X\cap Y|<2$
        \item $X$ form a clique in $TJ_k(G)$, $(\forall X\in \mathbf{Msets}(\sfTJ_k(G))) $ 
        \item For each $u\in V(\sfTJ_k(G))$, there are exactly $k$ sets $X\in \mathbf{Msets}(\sfTJ_k(G))$ such that $u\in X$
        \item For each $uv\in E(\sfTJ_k(G))$, there exists one set $X\in \mathbf{Msets}(\sfTJ_k(G))$ such that $u,v\in X$
    \end{enumerate}
\end{corollary}

Next, we will examine the relation between graphs $\sfTS_{k-1}(G)$ and $\sfTJ_k(G)$ ($k = \omega(G)$). 
For each set $v$ in $V(\sfTS_{k-1}(G))$  For each set $v$ in $V(\sfTS_{k-1}(G))$  let $\mathbf{Expand}(v)$ be the set of $u$ in $V_{\sfTJ_k(G)}$ such that $v\subset u$, then $\mathbf{Expand}(v)$ is a clique in $\sfTJ_{k}(G)$ (since each set pair in $\mathbf{Expand}(v)$ will have intersection $v$).
\begin{proposition}
\label{prop:edge_in_TS_k-1}
    For each $w,r$ in $V(\sfTS_{k-1}(G))$: $\mathbf{Expand}(w),\mathbf{Expand}(r)$ has a common element if and only if $wr$ is in $E(\sfTS_{k-1}(G))$.
\end{proposition}
\begin{proof}
\begin{itemize}
    \item Firstly, we prove that for each $w\neq r$ in $V(\sfTS_{k-1}(G))$ if $\mathbf{Expand}(w),\mathbf{Expand}(r)$ has a common element, then $wr$ is in $E(\sfTS_{k-1}(G))$. 
	Let $x$ be any common element of $\mathbf{Expand}(w), \mathbf{Expand}(r)$; then $x$ has $k$ elements, $w,r$ are its subsets with $k-1$ elements. 
	Because $(w\cup r)\subset x$ and $w\cup r$ has at least $k-1$ (because $w,r$ has $k-1$ elements). 
	Then $|w\cup r|=k$ or $|w\cup r|=k-1$. 
	For the latter case, $|w\cup r|=|w|=|r|$, which mean $w=w\cup r=r$ and this contradicts our hypothesis that $w\neq r$. 
	Hence, the case $|w\cup r|=k=|x|$ must occur, and hence $|w\cap r|=|w|+|r|-|w\cup r|=k-2$. Let $w=(w\cap r)+a$ and $r=(w\cap r)+b$. 
	Then, because $a,b\in x$ (which is a clique of size $k$ in $G$) we have $ab\in G$, which indicates $wr\in TS_{k-1}(G)$.
    \item Secondly, we prove that if $wr\in E(TS_{k-1}(G))$ then $\mathbf{Expand}(w),\mathbf{Expand}(r)$ has a common element. If $wr \in E(TS_{k-1}(G))$ then $w=(w\cap r)+a,r=(w\cap r)+b$ and $ab\in G$. Hence, $(w\cap r)+a+b$ should also form a clique of size $k$ in $G$, which is the common element of $\mathbf{Expand}(w),\mathbf{Expand}(r)$.
\end{itemize}
\end{proof}

\begin{corollary}
\label{cor:independent_points_in_TS}
The vertices $w\in V(TS_{k-1}(G))$ such that $\mathbf{Expand}(w)=\emptyset$ are isolated vertices. 
\end{corollary}

\begin{proposition}
\label{prop:Msets_equal}
     Given $k=\omega(G)$. Let $\mathbf{Msets}$ be the multiset of non-empty sets $\mathbf{Expand}(w)$, $(\forall w\in V(TS_{k-1}(G)$) then $Msets=\mathbf{Msets}(TJ_k(G))$.
\end{proposition}
\begin{proof}

    We have to prove that for every subset $x$ of $V(TJ_k(G))$, the number of occurrences of $x$ in $\mathbf{Msets}$ and $\mathbf{Msets}(TJ_k(G))$ equals. 
    \begin{itemize}
        \item If $x=\emptyset$ then the number of occurrences of $x$ in $\mathbf{Msets}$ and $\mathbf{Msets}(TJ_k(G))$ are both 0.
        \item If $x$ has at least 2 elements. Let any 2 distinct elements of $x$ be $u,v$. It suffices for us to prove that $\mathbf{Msets}$ can only contain at most one occurrence of $x$ and $x\in \mathbf{Msets}$ if and only if $x\in \mathbf{Msets}(TJ_k(G))$. 
        \begin{itemize}
            \item Firstly, we prove that $\mathbf{Msets}$ cannot contain 2 occurrences of $x$. Assume we have two occurrences of $x$ in $U$, let $w,r$ be any two different vertices of $TS_{k-1}(G)$ such that $\mathbf{Expand}(w)=\mathbf{Expand}(r)=x$. Because $w,r$ are two different sets with $k-1$ elements then $|w\cup r| \geq k$ however $u,v$ are two different sets with $k$ elements so $|u\cap v|\leq k-1$. This creates a contradiction because $(w\cup r)\subset (u\cap v)$ hence the hypothesis cannot happen and there should be at most one occurrence of $x$ in $U$.
            \item Secondly, we prove that if $x\in \mathbf{Msets} $ then $x\in \mathbf{Msets}(TJ_k(G))$. Assume $x=\mathbf{Expand}(w)$ then we have $w=(u\cap v)$ and hence $uv\in E(TJ_k(G))$. Using this fact, there must exist an index $i$ that $v\in S_i(u)$, and we will prove that $M_i(u)=x$ which make $x\in \mathbf{Msets}(TJ_k(G))$. We know that $M_i(u)=(N_{TJ_k(G)}(u)\cap N_{TJ_k(G)}(v))\cup \{u,v\}$(the set contain $u,v$ and all vertices incident with both $u,v$). The set of all vertex incidents with both $u,v$ will create a triangle in $TJ_k(G)$ is the set of vertex other than $u,v$ contain $w=(u\cap v)$ and hence $M_i(u)=x$.
            \item Finally, we prove that if $x\in \mathbf{Msets}(TJ_k(G))$ then $x\in \mathbf{Msets}$. Because $u,v\in x$ so $x$ is the set which contains $u,v$ and all vertices incident with both $u,v$. Let $w=(u\cap v)$ then $w\in V(TS_{k-1}(G))$ because $uv$ is an edge in $E(TJ_k(G))$. This makes $x$ the set of vertices contain $w$ and hence $x=\mathbf{Expand}(w)$.
        \end{itemize}
        \item If $x$ has only one element, let $X=\{u\}$. Because there are $k$ subsets $w$ with size $k-1$ of $u$, there are also $k$ sets $\mathbf{Expand}(w)$ that contain $u$. Therefore, there are $k$ sets in $\mathbf{Msets}$ that contain $u$, which is equivalent to $\mathbf{Msets}(TJ_k(G))$. Since we have proven that the number of occurrences of every set with at least two elements is the same, this implies that the number of occurrences of $\{u\}$ in the sets $\mathbf{Msets}$ and $\mathbf{Msets}(TJ_k(G))$ is also the same.
\end{itemize}
\end{proof}
For the next proposition, we use $H + cK_{1}$ to denote the graph obtained from $H$ by adding $c$ isolated vertices.
\begin{proposition}\label{prop:construct-TSk-1}
Given $k=\omega(G)$ and $\sfTJ_k(G)$ as input without knowing the set label of $\sfTJ_k(G)$, one can construct a graph $H$ in $k^2 \cdot poly(|\sfTJ_k(G)|)$ time such that $\sfTS_{k-1}(G) \cong H + cK_1$ for some integer $c \geq 0$.
\end{proposition}
\begin{proof}
    We call graph $T=(V,E)$ for short for our input graph $\sfTJ_k(G)$.
	Our algorithm will be split into 2 phases. 
	First phase: We want to check if the graph is $k$-good, that is, for each $u\in V$, we can partition $N_T(u)$ into $k$ sets $S_1(u), S_2(u), \dots, S_k(u)$.
    \begin{algorithm}
	\KwIn{A graph $T = \sfTJ_k(G)$}
	\KwOut{For each $u \in V(T)$, a collection of $k$ sets $S_1(u), S_2(u), \dots, S_k(u)$ such that each $S_i(u)$, if non-empty, is a clique of $T$, and no edge of $T$ joining a vertex of $S_i(u)$ and that of $S_j(u)$, where $i \neq j$ and $1 \leq i, j \leq k$}
	\SetArgSty{textbb} 
	\DontPrintSemicolon
    
    \For{$u\in V$}
    {   
        $Remain \gets N_T(u)$\;
        \For{$i=1$ \KwTo $k$}
        {
            \If{$Remain$ is empty}
            {
                $S_i(u) \gets \emptyset$\;
            }
			\Else{
                $v$ be an element of $Remain$\;
                $S_i(u)$ be the sets of elements in $Remain$ that are incident to $v$ and $v$ itself\;
                $Remain \gets Remain-S_i(u)$\;
                \If{$S_i(u)$ does not form a clique}{
                    report $T$ is not $k$-good \;
                }
                \ElseIf{exist an edge from $S_i(u)$ to $Remain$}{
                    report $T$ is not $k$-good \;
                }
            }
        }
        \If{$Remain$ is not empty}{
            report $T$ is not $k$-good \;
        }
    }

	\caption{Algorithm to check if a graph is $k$-good  and partition the neighbour set of each vertex.}
	\label{algo:1}
    \end{algorithm}
    
    \begin{algorithm}
    
	\KwIn{A $k$-good graph $T = \sfTJ_k(G)$ for some graph $G$ without any set label, and for each $u \in V(T)$, a collection of $k$ sets $S_1(u), S_2(u), \dots, S_k(u)$ obtained from Algorithm~\ref{algo:1}}
	\KwOut{A graph $H$ satisfying \cref{prop:construct-TSk-1}}

	\SetArgSty{textbb} 
	\DontPrintSemicolon

    $Msets \gets \emptyset$\;
    \For{$u\in V$}
    {   
        \For{$i=1$ \KwTo $k$}
        {
            $M_i(u) \gets S_i(u)\cup \{u\}$\;
            \eIf {$M_i(u)=\{u\}$}
            {
                $Msets \gets Msets+\{u\}$\;
            }{
                \If{$M_i(u) \notin Msets$}
                {
                    $Msets \gets Msets+M_i(u)$\;
                }
            }
        }
    }
    Edge set $E_H \gets \emptyset$\;
    \For{$U,V \in Msets$}
    {
        \If{$|U\cap V| \neq \emptyset$}
        {
            $E_H \gets E_H \cup \{U,V\}$\;
        }
    }
    \Return graph $H=(Msets,E_H)$\;

	\caption{Algorithm to obtain a graph $H$ satisfying \cref{prop:construct-TSk-1}.}
	\label{algo:2}
    \end{algorithm}
    
Second phase: We initialize multiset $\mathbf{Msets}$ and insert each of $M_i(u)$ if it has only one element or is not contained. 
Next, we will create a new graph $H$ with vertex set being the multiset $\mathbf{Msets}$ and the edges will be between those two set that have common elements.

$\mathbf{Evaluation:}$ The running time of first algorithm is at most $|N_T(u)|^2$ for each vertex $u\in V$ and $i\in [1,k]$. This means, in total, the running time is at most $k\cdot \sum_{u\in V}|N_T(u)|^2$ which is $k\cdot poly(size(T))$. For the second algorithm, the checking if $M_i(u)\in Msets(u\in V, i \in [1,k])$ can be done in $O(size(T))$ since we can check if there is any vertex in $M_i(u)$ that has been iterated and also then building the $E_H$ take at most $k^2\cdot poly(size(T))$ since $Msets$ has at most $k\cdot poly(size(T))$ elements. Hence, in total, we need a running time of $O(k^2\cdot poly(size(T)))$.

$\mathbf{Justification:}$ Firstly, we prove that Algorithm~\ref{algo:1} can be used to check for $k$-good graph and also split $N_T(u)$ into $k$ disjoint sets that form a clique. 
\begin{itemize}
    \item  If Algorithm~\ref{algo:1} runs successfully, then for every $u \in V$ we can partition $N_T(u)$ into $k$ disjoint sets $S_1(u), S_2(u), \dots, S_k(u)$ such that each form a clique, and there is no edge between different sets. 
	Hence, $T$ is a $k$-good graph and the partition of it is valid.
    \item If Algorithm~\ref{algo:1} fails somewhere, we consider the following cases.
	\begin{itemize}
        \item If it fails because $S_i(u)$ does not form a clique. This means there are $u,v, a,b$ in $V$ such that $uv,ua, ub,va,vb\in E$ but not $ab$. This means that $T$ is not $k$-good since if it is then $a,b$ cannot be in the same clique in the partition of $N_T(u)$ but $a,v$ and $v,b$ must be in the same clique which creates a contradiction.
        \item If it fails because $S_i(u),Remain$ have an edge between them. This means there are $u,v,a,b$ in $V$ such that $uv,ua, ub,va,ab\in E$ but not $vb$. This means that $T$ is not $k$-good because of the same explanation as above.
        \item If it fails because after the loop $Remain$ is not empty. Then the induced subgraph of $T$ with vertex set $N_T(u)$ have at least $k+1$ connected components which does not apply for $k$-good graph. Hence, $T$ is not $k$-good.
    \end{itemize}
\end{itemize}
Next, the second algorithm follows the idea of \cref{prop:Msets_equal} and \cref{prop:edge_in_TS_k-1} to build $\sfTS_{k-1}(G)$. 
However, there are several isolated points because of \cref{cor:independent_points_in_TS}.
\end{proof}

In the next part, for some integer $k$, recall that a graph $T$ satisfying that $T \cong \sfTJ_k(G)$ is called a \textit{$\sfTJ_k$-graph}. 
Additionally, we call $T$ a \textit{maximum $\sfTJ_k$-graph} if it is a $\sfTJ_{k}$-graph for some graph $G$ satisfying $k = \omega(G)$.
\begin{proposition}
\label{prop:maximal_TJ_graphic}
    For an integer $k$, if $T$ is a maximum $\sfTJ_k$-graph then it is a (maximum) $\sfTJ_n$-graph for any $n\geq k$.
\end{proposition}
\begin{proof}
    Let $G$ be the graph such that $\omega(G)=k$ and $\sfTJ_k(G)\cong T$. 
	For a number $n\geq k$, let the graph $G_n$ be the graph $G$ but add a clique $K_{n-k}$ and all the edges between the vertices in $G$ and the clique. 
	We can see that $\omega(G_n)=k+(n-k)=n$ and every clique with size $n$ must be formed from a clique with size $k$ in $G$ and the clique $K_{n-k}$. 
	Hence, $\sfTJ_n(G_n)\cong \sfTJ_k(G)\cong T$ which means $T$ is a maximum $\sfTJ_n$-graph. 
\end{proof}

\begin{proposition}
\label{prop:infinite_TJ_graphic}
    If for infinitely many integers $n$, a non-empty graph $T$ is a $\sfTJ_n$-graph then there exists $k$ so that $T$ is a maximum $\sfTJ_k$-graph. 
\end{proposition}
\begin{proof}
    Let $m=\omega(T)$. 
	Since $T$ is a $\sfTJ_n$-graph for infinitely many integers $n$, there must be a number $N>m$ such that $T$ is a $\sfTJ_N$-graph.
	Let $G$ be a graph such that $T \cong \sfTJ_N(G)$. 
	If $\omega(G)>N$, then $\omega(T)=\omega(\sfTJ_N(G))\geq N+1>m$ (\cref{cor:maxclique_TJ}), which creates a contradiction. 
	Hence, $\omega(G)=N$ (since $T$ is non-empty, $\omega(G)<N$ cannot happen). 
	Hence, $\omega(G)=N$ and hence $T$ is a maximum $\sfTJ_N$-graph.
\end{proof}

Thus, because of \cref{prop:maximal_TJ_graphic} and \cref{prop:infinite_TJ_graphic}, we have the following corollary.
\begin{corollary}
    A non-empty graph $T$ is a $\sfTJ_n$-graph for infinitely many integers $n$ if and only if there exists $k$ so that graph $T$ is a maximum $\sfTJ_k$-graph. 
\end{corollary}

\begin{proposition}
\label{prop:diamond_free_TJ_graphic}
    For $k\leq 3$, if a non-empty graph $T$ is a $\sfTJ_k$-graph and $T$ does not contain a diamond ($K_4-e$) as an induced subgraph then $T$ is a maximum $\sfTJ_k$-graph.
\end{proposition}
\begin{proof}
    If $k=1$, then the problem becomes trivial. The graph $T$ must be a graph $K_n$ for some $n$, and we can take the graph $H$ with $n$ vertices and no edges and show that $T\cong \sfTJ_1(H)$ and $\omega(H)=1$.
    
    If $k=2,3$, we consider the set of graphs $G$ such that $T\cong \sfTJ_k(G)$, we call this set $M$. 
	Here we refer each vertex of $\sfTJ_k(G)$ as a set of size $k$ that form a clique in $G$. 
	We can observe that the union $\cup_{v\in V(\sfTJ_k(G))}\{v\}$ is bounded by $k\cdot |V(\sfTJ_k(G))|=k \cdot |V(T)|$ elements. 
	So among the graph in $M$, we take a graph $G$ that produces it with the largest elements, we will prove that $\omega(G)=k$.
    \begin{itemize}
        \item If $k=2$ and $\omega(G)\geq 3$, there exists a triangle in $G$ and let any triangle in $G$ be $A,B,C$.
        \begin{itemize}
            \item We consider the case where there exists an edge in $G$ that is not in the triangle and connects to one of $A, B, C$. 
			Let's assume the edge is $AD$. 
			Then we can see the 4 edges $\{A, B\},\{A, C\},\{B, C\},\{A, D\}$ form a diamond in $\sfTJ_2(G)$, which cannot happen since $T$ does not have a diamond as an induced subgraph.
            \item If there is no edge other than $AB , BC, CA$ that connects to points $A, B, C$ in $G$, then we erase the edge $\{B,C\}$ and add a new vertex $D$ and connect it to $A$ to make a new graph $G_0$. We can prove that $\sfTJ_2(G)\cong \sfTJ_2(G_0)$ if we map the vertex in $V(\sfTJ_2(G))$ except for $\{B,C\}$ to its self in $V(\sfTJ_2(G_0))$ and map $\{B,C\}$ to $\{A,D\}$. So $\sfTJ_2(G_0)\cong T$ but also $|\cup_{v\in V(TJ_2(G_0))}\{v\}|=|\cup_{v\in V(TJ_2(G))}\{v\}|+1$ since there is a new point, $D$, this contradicts the definition of $G$.
        \end{itemize}
        \item If $k=3$ and $\omega(G)\geq 4$, there exists a $K_4$ in $G$, and let any $K_4$ in $G$ be $A,B,C,D$.
        \begin{itemize}
            \item If there is a triangle that is not a subset of $\{A, B, C, D\}$ and contains an edge of $\{A, B, C, D\}$, let's assume the triangle is $\{E, A, B\}$. 
			We consider the 4 triangles $\{E,A,B\}$, $\{A,B,C\}$, $\{A,B,D\}$, $\{A,C,D\}$; these will create a diamond in $\sfTJ_3(G)$. 
			This cannot happen since $T$ does not have an induced subgraph diamond.
            \item If there is no such triangle, we erase the edge $CD$ from $G$ and add two new vertices $E, F$ and connect them to $A, B$.  
			We can prove that $\sfTJ_3(G)\cong \sfTJ_3(G_0)$ if we map the vertex in $V(\sfTJ_3(G))$ except for $\{A, C, D\},\{B, C, D\}$ to its self in $V(\sfTJ_3(G_0))$ and map $\{A, C, D\}$ to $\{A,B,E\}$, $\{B, C, D\}$ to $\{A,B,F\}$. 
			So $\sfTJ_3(G_0)\cong T$ but also $|\cup_{v\in V(TJ_3(G_0))}\{v\}|=|\cup_{v\in V_{\sfTJ_3(G)}}\{v\}|+2$ since there are new points $D,E$, this contradicts the definition of $G$.
        \end{itemize}
    \end{itemize}
\end{proof}

\begin{corollary}
For an integer $k\leq 3$, if a graph is $T$ is a $\sfTJ_n$-graph for all $n\geq k$, then $T$ is a maximum $\sfTJ_k$-graph.
\end{corollary}
\begin{proof}
    Because of \cref{prop:infinite_TJ_graphic}, then there exist a number $N$ such that $T$ is a maximum $\sfTJ_N$-graph. Because of \cref{prop:maximal_TJ_is_diamond_free}, $T$ does not contain diamond as its induced subgraph. By \cref{prop:diamond_free_TJ_graphic}, $T$ is a maximum $\sfTJ_k$-graph. 
\end{proof}

\section{Planar $\sfTS_k$-Graphs}
\label{sec:planar-TSk}

This section is devoted to proving the following theorem.
\begin{theorem}\label{thm:TS-planar}
	If a graph $G$ is planar, then so is $\sfTS_k(G)$, where $1 \leq k \leq \omega(G) \leq 4$.
\end{theorem}
We remark that \cref{thm:TS-planar} is trivial when either $k = 1$ or $k = 4$.
When $k = 1$, $\sfTS_k(G) \cong G$ and therefore it is planar.
When $k = 4$, no two vertices of $\sfTS_k(G)$ are adjacent (otherwise, $G$ must have a $K_5$, which contradicts the assumption that it is planar), and thus it is again planar.

In the next lemma, we consider the case $k = 2$.
\begin{lemma}\label{lem:TS2-planar}
	If a graph $G$ is planar, then so is $\sfTS_2(G)$.
\end{lemma}
\begin{proof}
	Let $G_p$ be a planar drawing of $G$.
	We construct a graph $H$ from $G_p$ as follows.
	\begin{enumerate}[(1)]
		\item For each edge $uv$ of $G_p$, replace it by the path $u\alpha_{uv}v$ where $\alpha_{uv}$ is a newly added vertex inside line $uv$. We use $\alpha_{vu}$ and $\alpha_{uv}$ to refer to the same point.
		\item For each triangle $uvw$, we draw the edge between pair $\alpha_{uv},\alpha_{uw}$ by a zigzag line that is very close to $\alpha_{uv}-u-\alpha_{uw}$. Similarly, draw the edge between pairs $\alpha_{uv},\alpha_{vw}$ and $\alpha_{uw},\alpha_{wv}$.
		\item Remove all vertices and edges of $G_p$.
	\end{enumerate}
	
	
	Thus, $H$ is also a planar drawing (See \cref{app:detail_drawing_TS_2} for more details.).
	Additionally, since two size-$2$ cliques (i.e., edges) of $G$ are adjacent in $\sfTS_2(G)$ if and only if they belong to the same triangle, the bijective mapping $f: V(H) \to V(\sfTS_2(G))$ defined by $f(\alpha_{uv}) = \{u, v\}$ indeed satisfies $\alpha_{uv}\alpha_{xy} \in E(H)$ if and only if $f(\alpha_{uv})f(\alpha_{xy}) \in E(\sfTS_2(G))$.
	(Note that $V(H) = \{\alpha_{uv}: uv \in E(G)\}$.)
	Thus, $H \cong \sfTS_2(G)$ and therefore $\sfTS_2(G)$ is planar.
\end{proof}

To conclude our proof of \cref{thm:TS-planar}, we consider the case $k = 3$.
\begin{lemma}\label{lem:TS3-planar}
If a graph $G$ is planar, so is $\sfTS_3(G)$.
\end{lemma}
\begin{proof}
Here we can assume that $\omega(G)=4$.
Our first step is to prove that $\sfTJ_4(G)$ is an acyclic graph with maximum degree at most $4$ (see details in $\ref{app:property_TJ_4}$).
From the algorithm in \cref{prop:construct-TSk-1}, with the input graph $\sfTJ_4(G)$ (which is acyclic and has maximum degree at most $4$) and $k=4$, the output graph has all edges of $\sfTS_3(G)$ and isolated vertices. 
Thus, it is sufficient to show that running the algorithm for an acyclic graph with maximum degree at most $4$ and $k=4$ (i.e., the $\sfTJ_4(G)$ graph) will produce a planar graph. 

We prove this by induction on the number of vertices in the input graph.
Let $T=(V,E)$ be the input graph then $T$ will pass the first algorithm since for any vertex $X$ in $T$, because $deg(X)\leq 4$, we can put each element in $N_T(X)$ into a different set $S_i(X)$ ($1\leq i \leq 4$) and leave the remaining sets empty. 
Now we prove that $T$ will produce a planar graph in the second algorithm by induction on $|V|$:
\begin{itemize}
	\item $|V|=1$ then output will be a clique $K_4$.
	\item Assume that our hypothesis satisfies with any graph $T$ having vertices $<n$ vertices.
	We prove that any graph $T$ with $n$ vertices also satisfies the hypothesis. 
	For graph $T$ with no cycle, there is a vertex $X$ with a degree lower than $2$. Let $H$ be the graph $T$ without $X$ then $H$ is also a forest with maximum degree $\leq 4$. Let $P$ be the graph obtained by applying the algorithm with $H$. We know that $P$ is planar by induction hypothesis.
		\begin{itemize}
		\item If $deg(X)=0$ in $T$. Then applying the algorithm with $T$ will create $P$ and another $K_4$ with no edge between them. Then because $P, K_4$ are planar, we will also have a planar graph.
		\item If $deg(X)=1$ let $Y$ be the vertex that is adjacent with $X$ and let $Q$ be the graph obtained by applying algorithm for $T$.
		
		Compared to $H$, one of the empty sets $S_i(Y)$ now has $X$, therefore, the set $M_i(Y)$ changes from $\{Y\}$ to $\{Y,X\}$. 
		In addition, the sets $M_i(X)$ for $1 \le i \le 4$ are
$\{\{Y, X\}, \{X\}, \{X\}, \{X\}\}$, which means that we can obtain $Q$ by this: change one of the vertices in $P$ that has its label is $\{Y\}$ changed it to $\{Y,X\}$ and then add a $K_3$ that has label $\{X\}$ and connect them to $\{Y,X\}$. 

		This makes $Q$ still planar because we can draw this $K_3$  small in an angle of center $\{Y,X\}$ that does not contain any edge connecting this vertex and connect all the vertices of $K_3$ to it.
	\end{itemize}
\end{itemize}

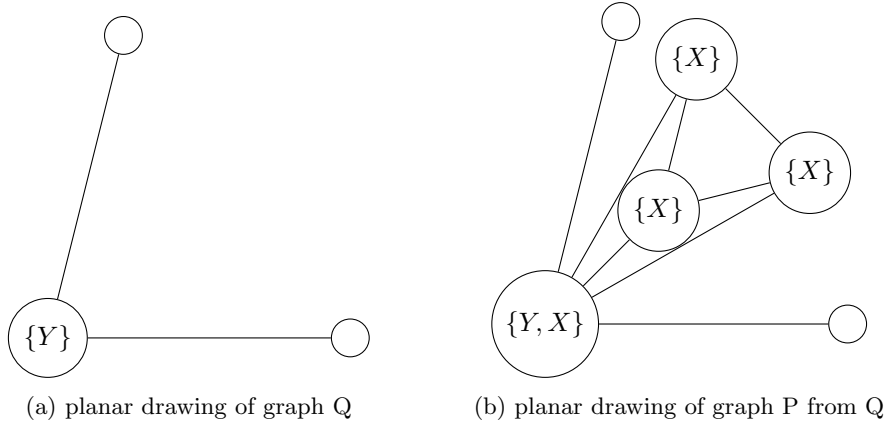
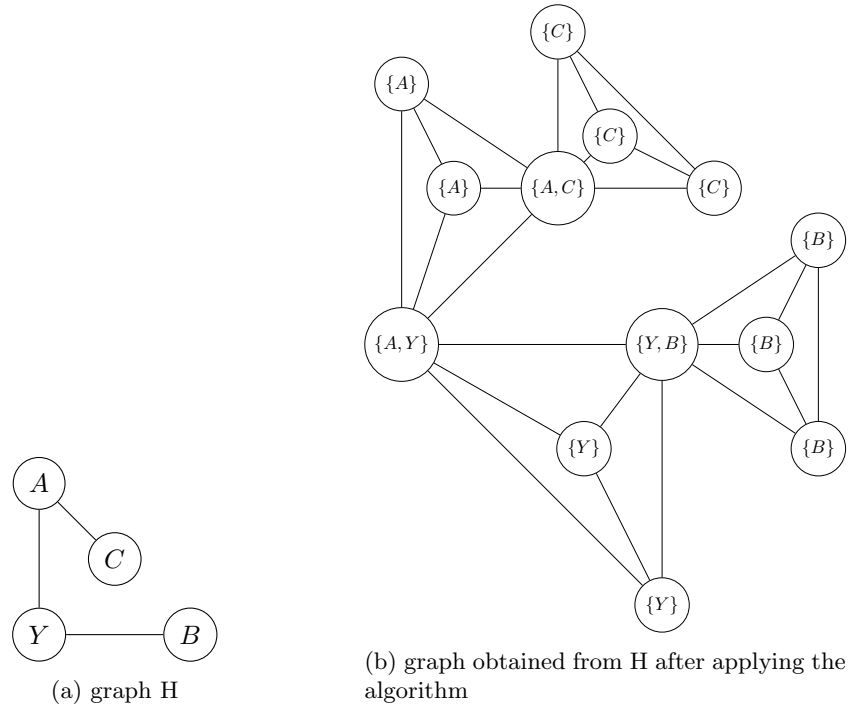
\begin{figure}[ht]
    \centering
    \begin{subfigure}[b]{0.4\textwidth}
    \centering
	\begin{adjustbox}{max width=\textwidth}
    \begin{tikzpicture}[every node/.style = {circle, draw, fill=white, minimum size=0.5cm}]
      \foreach \l/\x/\y/\c in {Y/0/0/{\{Y\}}, A/1/4/{}, B/4/0/{}}
      {
        \node (\l) at (\x, \y) {$\c$};
      }
      \draw (B)--(Y)--(A);
    \end{tikzpicture}
	\end{adjustbox}
    \caption{planar drawing of graph Q}
    \end{subfigure}
    \begin{subfigure}[b]{0.4\textwidth}
    \centering
	\begin{adjustbox}{max width=\textwidth}
    \begin{tikzpicture}[every node/.style = {circle, draw, fill=white, minimum size=0.5cm}]
      \foreach \l/\x/\y/\c in {Y/0/0/{\{Y,X\}}, A/1/4/{}, B/4/0/{}, X1/1.5/1.5/{\{X\}},X2/2/3.5/{\{X\}},X3/3.5/2/{\{X\}}}
      {
        \node (\l) at (\x, \y) {$\c$};
      }
      \draw (B)--(Y)--(A);
      \draw (X1)--(X2)--(X3)--(Y)--(X1)--(X3);
      \draw (X2)--(Y);
    \end{tikzpicture}
	\end{adjustbox}
    \caption{planar  drawing of graph P from Q}
    \end{subfigure}
	\caption{Illustrating Lemma \ref{lem:TS3-planar}.}
\end{figure}

\begin{figure}[ht]
    \centering
    \begin{subfigure}[b]{0.4\textwidth}
    \centering
	\begin{adjustbox}{max width=\textwidth}
    \begin{tikzpicture}[every node/.style = {circle, draw, fill=white, minimum size=0.5cm}]
      \foreach \l/\x/\y/\c in { Y/1/1/{Y}, A/1/3/{A}, B/3/1/{B} , C/2/2/{C}}
      {
        \node (\l) at (\x, \y) {$\c$};
      }
      
      \draw (Y)--(A)--(C);
      \draw (Y)--(B);
    \end{tikzpicture}
	\end{adjustbox}
    \caption{graph H}
    \end{subfigure}
    \begin{subfigure}[b]{0.4\textwidth}
    \centering
	\begin{adjustbox}{max width=\textwidth}
    \begin{tikzpicture}[every node/.style = {circle, draw, fill=white, minimum size=0.5cm}]
      \foreach \l/\x/\y/\c in {C1/4/4/{\{C\}}, C2/6/3/{\{C\}}, C3/3/6/{\{C\}}, AC/3/3/{\{A,C\}}, A1/1/3/{\{A\}}, AY/0/0/{\{A,Y\}}, A2/0/5/{\{A\}} , YB/5/0/{\{Y,B\}}, Y1/3.5/-2/\{Y\}, Y2/5/-5/{\{Y\}}, B1/7/0/{\{B\}},B2/8/2/{\{B\}},B3/8/-2/{\{B\}}}
      {
        \node (\l) at (\x, \y) {$\c$};
      }
      \draw (C1)--(C2)--(C3)--(AC)--(C1)--(C3);
      \draw (C2)--(AC);
      \draw (A1)--(A2)--(AY)--(AC)--(A1)--(AY);
      \draw (A2)--(AC);
      \draw (Y1)--(YB)--(AY)--(Y2)--(Y1)--(AY);
      \draw (YB)--(Y2);
      \draw (B1)--(B2)--(B3)--(YB)--(B1)--(B3);
      \draw (YB)--(B2);
    \end{tikzpicture}
	\end{adjustbox}
    \caption{graph obtained from H after applying the algorithm}
    \end{subfigure}
	\caption{Illustrating Lemma \ref{lem:TS3-planar}.}
\end{figure}

\newpage
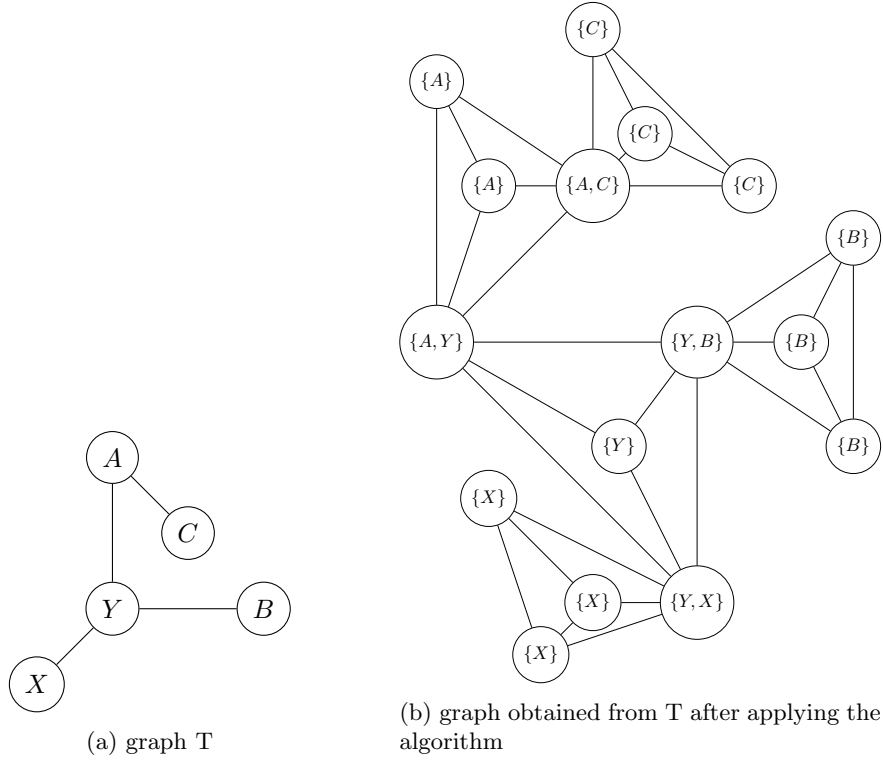
\begin{figure}[hbt!]
    \centering
    \begin{subfigure}[b]{0.4\textwidth}
    \centering
    \begin{tikzpicture}[every node/.style = {circle, draw, fill=white, minimum size=0.5cm}]
      \foreach \l/\x/\y/\c in {X/0/0/{X}, Y/1/1/{Y}, A/1/3/{A}, B/3/1/{B} , C/2/2/{C}}
      {
        \node (\l) at (\x, \y) {$\c$};
      }
      \draw (X)--(Y);
      \draw (Y)--(A)--(C);
      \draw (Y)--(B);
    \end{tikzpicture}
    \caption{graph T}
    \end{subfigure}
    \begin{subfigure}[b]{0.4\textwidth}
    \centering
	\begin{adjustbox}{max width=\textwidth}
    \begin{tikzpicture}[every node/.style = {circle, draw, fill=white, minimum size=0.5cm}]
      \foreach \l/\x/\y/\c in {C1/4/4/{\{C\}}, C2/6/3/{\{C\}}, C3/3/6/{\{C\}}, AC/3/3/{\{A,C\}}, A1/1/3/{\{A\}}, AY/0/0/{\{A,Y\}}, A2/0/5/{\{A\}} , YB/5/0/{\{Y,B\}}, Y1/3.5/-2/\{Y\}, YX/5/-5/{\{Y,X\}}, B1/7/0/{\{B\}},B2/8/2/{\{B\}},B3/8/-2/{\{B\}}, X1/3/-5/{\{X\}}, X2/1/-3/{\{X\}},X3/2/-6/{\{X\}}}
      {
        \node (\l) at (\x, \y) {$\c$};
      }
      \draw (C1)--(C2)--(C3)--(AC)--(C1)--(C3);
      \draw (C2)--(AC);
      \draw (A1)--(A2)--(AY)--(AC)--(A1)--(AY);
      \draw (A2)--(AC);
      \draw (Y1)--(YB)--(AY)--(YX)--(Y1)--(AY);
      \draw (YB)--(YX);
      \draw (B1)--(B2)--(B3)--(YB)--(B1)--(B3);
      \draw (YB)--(B2);
       \draw (X1)--(X2)--(X3)--(YX)--(X1)--(X3);
      \draw (X2)--(YX);
    \end{tikzpicture}
	\end{adjustbox}
    \caption{graph obtained from T after applying the algorithm}
    \end{subfigure}
	\caption{Illustrating Lemma \ref{lem:TS3-planar}.}
\end{figure}
\end{proof}

\newpage

\begin{corollary}
    For a planar graph $G=(V,E)$, let $F_3,F_4$ be the number of $K_3,K_4$ in graph $G$ then we have $F_3\leq |E|-2$ and $ 2\cdot F_4 \leq F_3-2$. Moreover, $F_3\leq 3\cdot|V|-8$.
\end{corollary}
\begin{proof}
    For a planar graph $T=(V_T,E_T)$, the bound for the number of edge is : $|E_T|\leq 3\dot|V_T|-6$. Apply this bound to $TS_2(G)$ and $TS_3(G)$. Finally, apply the bound to $G$ to get $F_3\leq 3\cdot|V|-8$.
\end{proof}
\section{Some Well-Known Properties of Simplex Graphs}
\label{sec:simplex}

In this section, we mention some well-known properties of simplex graphs ($\sfTAR$-graphs). More details can be found in~\cite{BandeltV89,KlavzarM02,CabelloEK11,ScharpfeneckerT16}. Recall that a graph $G$ is called a \textit{median graph} if every triple $a, b, c \in V(G)$ of vertices has a unique \textit{median} $m(a, b, c)$---a vertex that belongs to shortest paths between each pair of $a$, $b$, and $c$.

\begin{proposition}[\cite{BandeltV89}]\label{prop:simplex-is-median}
	Any simplex graph is a median graph.
\end{proposition}

It is well-known that any median graph is bipartite and by \cref{prop:simplex-is-median} so is any simplex graph. 

Recall that the join $G \oplus H$ of two graphs $G$ and $H$ is the graph obtained from the disjoint union of $G$ and $H$ by joining every vertex of $G$ with every vertex of $H$.
The Cartesian product of two graphs $G$ and $H$, denoted by $G \square H$, is the graph with vertex-set $V(G) \times V(H)$ and $(a, x)(b, y) \in E(G \square H)$ whenever either $ab \in E(G)$ and $x = y$ or $a = b$ and $xy \in E(H)$.

\begin{proposition}[\cite{KlavzarM02}]
	Let $G$ and $H$ be two disjoint graphs.
	Then, $\sfTAR(G \oplus H) = \sfTAR(G) \square \sfTAR(H)$.
	In other words, the Cartesian product of two simplex graphs is also a simple graph.
\end{proposition}

The following proposition involves the simplex graphs of some simple graphs.
Recall that a \textit{gear graph} (also known as \textit{bipartite wheel graph}) is a graph obtained by inserting an extra vertex between each pair of adjacent vertices on the perimeter of a wheel graph.
A \textit{Fibonacci cube}~\cite{KlavzarM02} $\Gamma_n$ is a subgraph of a hypercube $Q_n$ induced by vertices (binary strings of length $n$) not having two consecutive $1$s.
\begin{proposition}
	\begin{enumerate}[(a)]
		\item Simplex graph of a complete graph is a hypercube.
		\item Simplex graph of a cycle of length $n \geq 4$ is a gear graph.
		\item Simplex graph of the complement of a path is a Fibonacci cube.
	\end{enumerate}
\end{proposition}

\begin{proof}
    \begin{enumerate}[(a)]
        \item Given a complete graph $K_n$, every subsets of the vertex set $\{v_1,v_2,...,v_n\}$ of $K_n$ forms a clique, including from 0-clique, which is the empty set, to $n-$clique. Hence, we can see that the simplex graph of $K_n$ has $2^n$ vertices. Similar to the hypercube $Q_n$, the vertices of $Q_n$ can be represented by binary strings of length $n$, which correspond to the $2^n$ subsets of the vertex set of $Q_n$. Then, we define a bijection
        \begin{align*}
        f:V(\sfTAR(K_n)) &\rightarrow V(Q_n)\\
        C &\mapsto f(C), \quad \text{ where $C$ is a clique in $K_n$}.
    \end{align*}
    It suffices to show that $f$ is an isomorphism since for all $C\in V(K_n)$, $f(C)$ is the corresponding binary string of length $n$, where the $i$th bit 1 if the $i$th vertex is included in $C$, and 0 otherwise. Furthermore, two cliques in the simplex graph of $K_n$ are adjacent if they differ by one element, corresponding to $Q_n$ since two vertices in $Q_n$ are adjacent if their binary strings differ by one bit. Then we can conclude that $K_n$ is isomorphic the hypercube $Q_n$.

    \item Given a cycle graph $C_n,n\geq 4$. Since the cliques in $C_n$ include all subsets of vertices that form complete subgraphs, we can conclude that the simplex graph of $C_n$ only has 0-clique, which is the empty set, 1-cliques, and 2-cliques. We can see that the number of 1-cliques and the number of 2-cliques in $C_n$ is equal to $n$, because the number of vertices in $C_n$ is equal to the number of edges. Together with the empty set, the total number of vertices in the simplex graph of $C_n$ is $2n+1$. Then, every 1-cliques is adjacent with the empty set, and each 2-cliques is adjacent with 2 1-cliques which are adjacent in $C_n$, hence the number of edges in the simplex graph of $C_n$ is $3n$. Since the simplex graph of $C_n$ has the equivalent number of vertices and edges to a general gear graph, we can define a bijection
    \begin{align*}
        f:V(\sfTAR(C_n)) &\rightarrow V(G_n), \quad \text{ where $G_n$ is the gear graph.}
    \end{align*}
    It suffices to show that $f$ is an isomorphism since for any 1-clique and 2-clique vertices in the $\sfTAR(C_n)$ can be seen as the outer cycle of the gear graph $G_n$, and the empty set vertex is the central vertex of $G_n$. Besides, each 1-clique vertex in $\sfTAR(C_n)$ is connected to the empty set vertex, and each 2-clique vertex is connected to 1-clique vertices, which form a gear graph $G_n$.

    \item A proof of this statement appeared in~\cite{KlavzarM02}. 
    

    \end{enumerate} 
\end{proof}

\section*{Acknowledgment}
This work began during the 2023--2024 edition of the Vietnam Polymath REU (VPR) program.  
We thank the organizers and participants of VPR for creating a fruitful research environment.
Duc A. Hoang's research was partially supported by the Vietnam Institute for Advanced Study in Mathematics (VIASM), and the Vietnam National University, Hanoi under the project QG.25.07 ``A study on reconfiguration problems from algorithmic and graph-theoretic perspectives''. Finally, we would like to thank the anonymous reviewers of the International Conference on Algorithms and Discrete Applied Mathematics for their valuable comments and suggestions that helped improve this paper.


\printbibliography

\appendix

\section{Planar Drawing of $\sfTS_2(G)$ in \cref{lem:TS2-planar}}\label{app:detail_drawing_TS_2}

In this section, we describe a detailed construction of a planar drawing of the graph $TS_2(G)$, where $G$ is any given planar graph. (\cref{lem:TS2-planar}.)
\begin{tcolorbox}[title=Main idea for drawing edges for graph $TS_2(G)$]
Let $G_p$ be a planar drawing of $G$.
Each edge in $TS_2(G)$ has the form $\alpha_{uv}\alpha_{uw}$ where $uv, uw, vw$ are edges of $G_p$. We add a new vertex $\alpha_{u,vw}(\neq u)$ lie in side $\triangle uvw$ and close enough to $u$ and draw the edge from $\alpha_{uv}$ to $\alpha_{u,vw}$ and to $\alpha_{uw}$(we can imagine the line is drawn really close and along the edges $uv,uw$)
\end{tcolorbox}
The drawing based on these two lemma for the planar drawing $G_p$.
\begin{lemma}
    If there are 2 triangles $\triangle ABC$ and $\triangle ADE$ ($B,C,D,E$ are not necessarily different) on $G_p$ then $\angle{BAC},\angle{DAE}$ do not intersect or one of them has to contain the other.
\end{lemma}
\begin{lemma}
    Given that $\angle{BAC}$ ($B\neq C$) contains $\angle{DAE}$( $D\neq E$) ( but $B,C,D,E$ are not necessarily different), and the points $\alpha_{AB},\alpha_{AC},\alpha_{AD},\alpha_{AE}$ lie on the edge $AB,AC,AD,AE$, respectively. Suppose there is a point $\alpha_{A,DE}\neq A$ inside triangle $\triangle ADE$ then there exist $d >0$ such that for every $\alpha_{A,BC}\neq A$ inside triangle $\triangle ABC$ such that if the distance of $A$ and  $\alpha_{A,BC}$ smaller than $d$  then the edge $\alpha_{AB}-\alpha_{A,BC}-\alpha_{AC}$(zigzag line, joint of line $\alpha_{AB}$ to $\alpha_{A,BC}$ and $\alpha_{A,BC}$ to $\alpha_{AC}$) and $\alpha_{AD}-\alpha_{A,DE}-\alpha_{AE}$(zigzag line, joint of line $\alpha_{AD}$ to $\alpha_{A,DE}$ and $\alpha_{A,DE}$ to $\alpha_{AE}$) does not intersect other than the endpoints. 
\end{lemma}
\begin{proof}
    The proof of the first lemma is trivial to see. For the second lemma, we will find $d_1$ such that if the distance of $A$ and  $\alpha_{A,BC}$ smaller than $d_1$ then $\alpha_{AB}-\alpha_{A,BC}$ does not intersect $\alpha_{AD}-\alpha_{A,DE}-\alpha_{AE}$ in the inside or at the endpoint $\alpha_{A,BC}$. We will want ray $\alpha_{AB}-\alpha_{A,BC}$ to lie outside angles made by pairs of rays $\alpha_{AB}-\alpha_{AD}$,$\alpha_{AB}-\alpha_{A,DE}$ and $\alpha_{AB}-\alpha_{AE}$,$\alpha_{AB}-\alpha_{A,DE}$ (if $\alpha_{AB}=\alpha_{AE}$ the ray $\alpha_{AB}-\alpha_{AE}$ will be considered $\alpha_{AB}-\alpha_{A,DE}$ and same goes for $\alpha_{AB}=\alpha_{AD}$). Because all points $\alpha_{AD},\alpha_{AE},\alpha_{A,DE}$ lies on the same half plane divided by $AB$, ray $\alpha_{AB}-A$ lie outside of those angles. Hence for $d_1$ small enough, $\alpha_{AB}-\alpha_{A,BC}$ become close to $\alpha_{AB}-A$ and lie outside those angles. Similarly, we will find $d_2$ such that if the distance of $A$ and  $\alpha_{A,BC}$ is smaller than $d_2$ then $\alpha_{AC}-\alpha_{A,BC}$ does not intersect in the inside or at the endpoint $\alpha_{A,BC}$ and we just need $d=min(d_1,d_2)$.
\end{proof}

We can see that the points of form $\alpha_{ab}$ split the edge $ab$ into two lines $\alpha_{ab}-a$ and $\alpha_{ab}-b$ and each pair of lines of that form also do not intersect in the inside. Let $r(>0)$ be the smallest distance between the non-intersecting pairs. Returning to the drawing, for each vertex $u$, we will draw edge $\alpha_{uv}\alpha_{uw}$ based on the angle $\angle{wuv}$ (from small to large). When we iterate to triangle $\triangle{uvw}$, we take a point $\alpha_{u,vw}$ close enough to $u$ so that joints line $\alpha_{uv}-\alpha_{u,vw}-\alpha_{uw}$ do not intersect all the previous edges of the form $\alpha_{ua}-\alpha_{u,ab}-\alpha_{ub}$ such that angle $\angle{aub}$ is inside $\angle{vuw}$. . We can maintain drawing in the manner that the distance between $u$ and $\alpha_{u,vw}$ is smaller than $\frac{r}{2}$ and the angle $\angle{u\alpha_{u,vw}\alpha_{uv}}$,$\angle{u\alpha_{u,vw}\alpha_{uw}}>90^\circ
$ 

It remains to validate the drawing. From this point, we use $dis(A,l)$ to denote the distance between a vertex $A$ and a line $l$. Assume the drawing is not planar, then there exist 2 different triangles $\triangle uvw$ and $\triangle abc$ such that $\alpha_{uv}-\alpha_{u,vw}-\alpha_{uw}$ and $\alpha_{ab}-\alpha_{a,bc}-\alpha_{ac}$ intersect in the inside. Assume that they intersect at $t$. There are 2 cases to consider:
\begin{itemize}
    \item $u=a$. Since the ray $at$ is inside angle $\angle{bac}$ and $\angle{vaw}$ so one of them contain the other. Assume $\angle{bac}$ contain $\angle{vaw}$ this mean $\angle{bac}>\angle{vaw}$. However, we have already settled this case with our drawing. Hence the 2 joint lines can not intertsect in the inside.
    \item $u\neq a$. The line $\alpha_{uv}-u$ and $\alpha_{ab}-a$ will be non-intersecting and have the distance of at least $r$. Assume $\alpha_{uv}-\alpha_{u,vw}$ and $\alpha_{ab}-\alpha_{a,bc}$ intersect at $t$. Let $h_1,h_2$ be the foot of altitude from $t$ to $\alpha_{uv}u$ and from $t$ to $\alpha_{ab}a$, respectively.

    We know that $h_1,h_2$ lies inside $\alpha_{uv}-u$ and $\alpha_{ab}-a$,respectively since $\angle{u\alpha_{u,vw}\alpha_{uv}}$,$\angle{a\alpha_{a,bc}\alpha_{ab}}>90^\circ$. Hence,we can create a contradiction:
    $$r\leq h_1h_2\leq th_1+th_2\leq dis(\alpha_{a,bc},\alpha_{ab}-a)+dis(\alpha_{u,vw},\alpha_{uv}-u)\leq \alpha_{a,bc}a+\alpha_{u,vw}u<r$$

\end{itemize}
Hence, the drawing is planar.

\section{Property of $\sfTJ_4(G)$}\label{app:property_TJ_4}
\begin{lemma}
For a planar graph $G$, graph $\sfTJ_4(G)$ is acyclic and have maximum degree at most $4$.    
\end{lemma}
\begin{proof}
    We refer $G$ to a graph with straight edge such that each pair of edges does not intersect inside. Let's assume that $TJ_4(G)$ has a cycle $C_n$ (suppose the set corresponding to the vertices in the cycle are $A_0, A_1, \ldots, A_{n-1}$). Let's assume that $A_1 = \{B, C, D, E\}$ has the convex hull with the minimum area among these sets, and let the triangle $BCD$ be its convex hull.

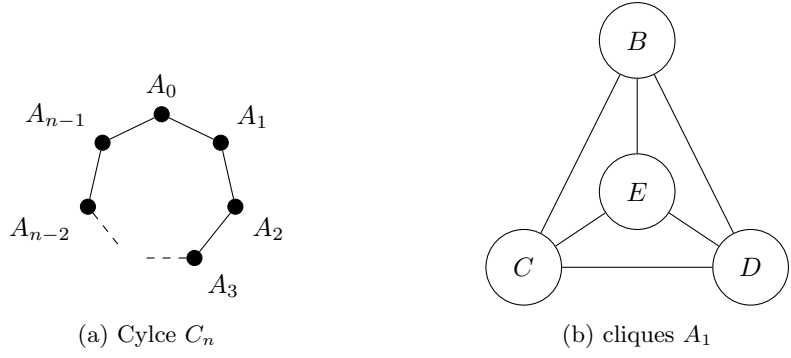
\begin{figure}[hbt!]
\centering
\begin{subfigure}{0.4 \textwidth}
	\centering
	\begin{tikzpicture}[dot/.style={draw,fill,circle,inner sep=2pt}]
\foreach \l [count=\n] in {0,1,2,3,{},n-2,n-1} {
\pgfmathsetmacro\angle{90-360/7*(\n-1)}
\ifnum\n=5
	\coordinate (n\n) at (\angle:1);
\else
	\node[dot,label={\angle:$A_{\l}$}] (n\n) at (\angle:1) {};
\fi
}
\draw (n6) -- (n7) -- (n1) -- (n2) -- (n3) -- (n4);
\draw[dashed,shorten >=5pt] (n4) -- (n5);
\draw[dashed,shorten >=5pt] (n6) -- (n5);
\end{tikzpicture}
\caption{Cylce $C_n$}
\end{subfigure}
\begin{subfigure}[b]{0.4\textwidth}
\centering
\begin{tikzpicture}[every node/.style = {circle, draw, fill=white, minimum size=1cm}]
	\foreach \l/\x/\y/\c in {C/0/0/{C}, D/3/0/{D}, B/1.5/3/{B}, E/1.5/1/{E}}
	{
	\node (\l) at (\x, \y) {$\c$};
	}
	\draw (B)--(C)--(D)--(B)--(E)--(D);
	\draw (C)--(E);
\end{tikzpicture}
\caption{cliques $A_1$}
\end{subfigure}  
\caption{Cycle in $\sfTJ_4(G)$ and element $A_1$ in the cycle}
\end{figure}

We prove that $A_0 \cap A_1 = \{B, C, D\}$. Assume otherwise. We must have $E \in (A_0 \cap A_1)$. Let $A_0 \setminus A_1 = {F}$. Then, there must be an edge between $E$ and $F$ in $G$. Since the edge $E-F$ does not intersect any edges $A-B$, $B-C$, or $A-C$, $F$ must be a vertex inside triangle $BCD$. Then, all vertices of $A_0$ must be inside triangle $BCD$, implying that the convex hull of set $A_0$ has an area less than $BCD$, which contradicts the assumption that $A_1$ has the convex hull with the minimum area among the cycle. Hence, $A_0 \cap A_1 = {B, C, D}$.

Similarly, we have $A_1 \cap A_2 = \{B, C, D\}$. Let's assume $A_0 = {B, C, D, F}$ and $A_2 = {B, C, D, H}$. If $F$ is inside triangle $BCD$, then it must be inside one of the three triangles $BCE$, $BDE$, or $CDE$, which means the edge between $F$ and the other vertices cuts one of the three edges in that triangle (for example, if $F$ is inside triangle $BCE$, then the edge $FD$ must cut one of the edges $BC$, $BE$, or $CE$). Hence, $F$ (and similarly $H$) is outside of triangle $BCD$. However, we cannot find such two vertices outside the triangle $BCD$ so that none of the edges $BF$, $CF$, $DF$, $BH$, $CH$, $DH$ intersect, indicating a contradiction, and we can prove that the graph has no cycle.

Now, assume that there is a vertex $A$ in $TJ_4(G)$ such that $\text{deg}(X) \geq 5$. Since $X$ only has 4 subsets of 3-elements, there must be 2 different vertices $Y$ and $Z$ incident with $X$ such that $X \cap Y = X \cap Z$. Hence, $(X \cap Y) \subset (Y \cap Z)$, and therefore, $Y \cap Z$ has exactly 3 elements and $Y,Z$ are adjacent in $TJ_4(G)$, indicating a contradiction. Hence, every vertex in $TJ_4(G)$ has a degree $\leq 4$.

\end{proof}

\end{document}